\documentclass{birkjour}

\usepackage[utf8]{inputenc}
\usepackage[T1]{fontenc}
\usepackage[english]{babel}
\usepackage{mathtools,amsmath,amssymb,amsfonts,amsthm,mathrsfs}
\mathtoolsset{showonlyrefs}
\usepackage{enumitem}

\usepackage{geometry}
\usepackage{graphicx,color}
\usepackage{tikz,pgfplots,pgf}
\usepackage[caption=false]{epsfig}
\usepackage{subcaption}
\usepackage[font=small]{caption}
\usepackage[pdftex]{hyperref}

\hypersetup{
  bookmarksnumbered, linkcolor=gray, citecolor=gray,
  bookmarks, breaklinks,
  linkbordercolor=gray,
  citebordercolor=gray,
}

\newtheorem{theorem}{Theorem}[section]
\newtheorem{lemma}{Lemma}[section]
\newtheorem{proposition}{Proposition}[section]
\newtheorem{corollary}{Corollary}[section]
\theoremstyle{definition}
\newtheorem{remark}{Remark}[section]

\newtheorem{example}{Example}[section]

\numberwithin{equation}{section}

\newcommand{\dom}{\mathop\mathrm{dom}\nolimits}
\renewcommand{\Im}{\mathop\mathrm{Im}\nolimits}

\begin{document}

\title[Atypical bifurcation for periodic solutions of $\phi$-Laplacian systems]{Atypical bifurcation for periodic solutions \\of $\phi$-Laplacian systems}

\author[P.~Benevieri]{Pierluigi Benevieri}

\address{
Instituto de Matem\'{a}tica e Estat\`{i}stica, Universidade de S\~{a}o Paulo\\
Rua do Mat\~{a}o 1010, S\~{a}o Paulo, SP - CEP 05508-090, Brazil}

\email{pluigi@ime.usp.br}

\author[G.~Feltrin]{Guglielmo Feltrin}

\address{
Department of Mathematics, Computer Science and Physics, University of Udine\\
Via delle Scienze 206, 33100 Udine, Italy}

\email{guglielmo.feltrin@uniud.it}

\thanks{Work written under the auspices of the Grup\-po Na\-zio\-na\-le per l'Anali\-si Ma\-te\-ma\-ti\-ca, la Pro\-ba\-bi\-li\-t\`{a} e le lo\-ro Appli\-ca\-zio\-ni (GNAMPA) of the Isti\-tu\-to Na\-zio\-na\-le di Al\-ta Ma\-te\-ma\-ti\-ca (INdAM). The first author is partially supported by GNAMPA. The second author acknowledges the support of INdAM-GNAMPA project ``Analisi qualitativa di problemi differenziali non lineari'' and PRIN Project 20227HX33Z ``Pattern formation in nonlinear phenomena''.
\\
\textbf{Preprint}} 



\date{}

\dedicatory{Dedicated to Professor Maria Patrizia Pera on the occasion of her 70th birthday}

\begin{abstract}
In this paper, we study the $T$-periodic solutions of the parameter-dependent $\phi$-Laplacian equation
\begin{equation*}
(\phi(x'))'=F(\lambda,t,x,x').
\end{equation*}
Based on the topological degree theory, we present some atypical bifurcation results in the sense of Prodi--Ambrosetti, i.e., bifurcation of $T$-periodic solutions from $\lambda=0$. Finally, we propose some applications to Li\'enard-type equations.
\end{abstract}

\maketitle

\section{Introduction and main result}\label{section-1}

The investigation of periodic boundary value problems of the form
\begin{equation}\label{eq-Mawhin}
\begin{cases}
\, x'=F(t,x),
\\
\, x(0)=x(T),
\end{cases}
\end{equation}
where $F\colon \mathopen{[}0,T\mathclose{]}\times \mathbb{R} \to \mathbb{R}$ is a given function, is a very classical topic in nonlinear analysis and, looking in the literature, it has received a lot of attention and has been tackled with several different techniques. 
Definitely, one of the most used and effective tools to the analysis of \eqref{eq-Mawhin} is the topological degree theory. 
With this respect, the continuation theorem by Mawhin is a milestones for the existence of periodic solutions, see \cite{CMZ-90,CMZ-92,Ma-69,Ma-93}. More precisely, Mawhin and collaborators prove the existence of solutions of \eqref{eq-Mawhin} exploiting the homotopy invariance property of the so-called \textit{coincidence degree}, introduced in \cite{Ma-72} (see also \cite{GaMa-77,Ma-79,Ma-93}) in order to extend the previous Leray--Schauder continuation principle \cite{Ma-99}.

In more details, Mawhin considers the parameter-dependent equation $x'=\vartheta F(t,x)$, with $\vartheta \in\mathopen{(}0,1\mathclose{]}$, to reduce the search of $T$-periodic solutions of \eqref{eq-Mawhin} to the study of the zeros of the finite-dimensional function $s \mapsto \int_{0}^{T} F(t,s) \,\mathrm{d}t$ in $\mathbb{R}$, cf.~\cite{Ma-69,Ma-93}. 
With a similar topological approach, in \cite{CMZ-92} the authors prove an analogous existence result introducing a homotopy with an autonomous field, thus looking for zeros of a finite-dimensional function, as before (see also \cite{BaMA-91}). 
Both results are also valid in the vector case and, in the last decades, a great number of applications to second-order periodic differential equations of the form
\begin{equation*}
x'' + f(t,x,x') = 0
\end{equation*}
has been provided. We refer the reader to \cite{Ma-23,Ma-05,MRZ-00} and the references therein.

More recently, several papers have considered the same topological approach dealing with periodic differential equations of the form
\begin{equation}\label{eq-intro-phi-L}
(\phi(x'))' + f(t,x,x') = 0,
\end{equation}
where $\phi$ is a homeomorphism. A significant contribution in this direction has been proposed by Man\'{a}sevich and Mawhin in \cite{MaMa-98}, where they extend the above mentioned continuation theorems obtaining periodic solutions to \eqref{eq-intro-phi-L} by an application of the Leray--Schauder degree theory. 
This result has been generalized in \cite{FeZa-17}, where some technical hypotheses on $\phi$ have been removed, via a slightly different point of view dealing with systems of first-order equations. We also mention \cite{GHMMT-24} for further recent investigations in this direction, and refer the reader to \cite{BdOdM-06,BeMa-08,BeMa-09} and the references therein, for existence results in the context of singular $\phi$-Laplacian equations, in particular the Minkowski curvature equations and Euclidean mean curvature equations.

In all the above mentioned contributions, a real parameter is used to reduce the original problem to a simpler one, exploiting the homotopy invariance property of the degree. 
The aim of the present paper is to propose a different point of view of the problem, taking into account the dependence of the solutions on the parameter. Accordingly, we modify system \eqref{eq-Mawhin} introducing a real parameter $\lambda$ as follows
\begin{equation}\label{eq-Mawhin-parameter}
\begin{cases}
\, x'=\lambda F(t,x),
\\
\, x(0)=x(T),
\end{cases}
\end{equation}
where $F\colon \mathopen{[}0,T\mathclose{]}\times \mathbb{R}^{n} \to \mathbb{R}^{n}$.
Observe that the $n$-dimensional space of real constant functions is the set of solutions of \eqref{eq-Mawhin-parameter} when $\lambda=0$. A quite natural question is whether one (or more) of such constant solutions is a bifurcation point of solutions with $\lambda\neq0$. 
Another problem concerns the topological behaviour of the sets of the solutions with $\lambda\neq0$ that emanate from a bifurcation point.

In this spirit, in this paper we consider the parametrized $T$-periodic boundary value problem
\begin{equation}\label{mainpb}
\begin{cases}
\, (\phi(x'))'=F(\lambda,t,x,x'),
\\
\, x(0)=x(T),\quad x'(0)=x'(T),
\end{cases}
\end{equation}
with $\lambda\in\mathbb{R}$, assuming the following conditions:
\begin{enumerate}[leftmargin=26pt,labelsep=6pt,label=\textup{$(\textsc{h}_{1})$}]
\item $\phi \colon\mathbb{R}^n\to \mathbb{R}^n$ is a homeomorphism with $\phi(\mathbb{R}^n)=\mathbb{R}^n$ and $\phi(0)=0$;
\label{hp-H1}
\end{enumerate}
\begin{enumerate}[leftmargin=26pt,labelsep=6pt,label=\textup{$(\textsc{h}_{2})$}]
\item $F\colon\mathbb{R}\times\mathopen{[}0,T\mathclose{]}\times\mathbb{R}^n\times\mathbb{R}^n\to\mathbb{R}^n$ is a Carath\'eodory function, that is,
\begin{itemize}[leftmargin=16pt,labelsep=6pt]
\item for almost every $t\in \mathopen{[}0,T\mathclose{]}$, $F(\cdot, t, \cdot,\cdot)$ is continuous,
\item for every $(\lambda,x,y) \in \mathbb{R}\times \mathbb{R}^n\times\mathbb{R}^n$, $F(\lambda,\cdot, x,y)$ is measurable,
\item for every $\rho>0$ and every compact interval $I$, there exists a map $g\in L^1(\mathopen{[}0,T\mathclose{]},[0,+\infty))$ such that, for almost every $t\in \mathopen{[}0,T\mathclose{]}$ and every $(\lambda,x,y) \in I\times \mathbb{R}^n\times\mathbb{R}^n$, with $\Vert x\Vert\leq \rho$ and  $\| y\|\leq \rho$, we have $\| F(\lambda,t,x,y)\|\leq g(t)$;
\end{itemize}
\label{hp-H2}
\end{enumerate} 
\begin{enumerate}[leftmargin=26pt,labelsep=6pt,label=\textup{$(\textsc{h}_{3})$}]
\item $F(0,t,x,y)=f_0(x,y)$, that is, $F(0,\cdot, \cdot, \cdot)$ is an autonomous continuous vector field. 
\label{hp-H3}
\end{enumerate}
A \textit{solution} of system \eqref{mainpb}, for a given $\lambda\in\mathbb{R}$, is a continuously differentiable map $x\colon \mathopen{[}0,T\mathclose{]}\to \mathbb{R}^n$ having absolutely continuous derivative and verifying the first equality in \eqref{mainpb}, for a.e.~$t\in \mathopen{[}0,T\mathclose{]}$, as well as the periodic boundary conditions. The pair $(\lambda,x)$ will be called \textit{solution pair} of \eqref{mainpb}.

In this paper, we establish a general bifurcation theorem (Theorem~\ref{genbifthm} below and subsequent corollaries) for system \eqref{mainpb}. We will study the so-called \textit{atypical bifurcation}, in the sense of Prodi--Ambrosetti \cite{PrAm-73}. 
Let us outline the key differences between the atypical bifurcation and the classical bifurcation theory.

Given two real Banach space $X$ and $Y$ and a continuous function $\mathcal{F} \colon \mathbb{R}\times X \to Y$, the bifurcation theory typically investigates problems of the form
\begin{equation}\label{eq-intro-biforcation}
\mathcal{F}(\lambda,x)=0, \quad (\lambda,x)\in \mathbb{R}\times X.
\end{equation}
Assume that $\mathcal{F}(\lambda,0)=0$ for every $\lambda\in\mathbb{R}$, i.e., $(\lambda,0)$ is a \textit{trivial} solution for every $\lambda\in\mathbb{R}$. The value $\lambda_{0}$ is called \textit{bifurcation point} if every neighbourhood of $(\lambda_{0},0)$ in $\mathbb{R}\times X$ contains solutions $(\lambda,x)$ of \eqref{eq-intro-biforcation} with $x\neq 0$. 
By the implicit function theorem, if $\mathcal{F}$ is Fr\'echet differentiable at $x=0$ and $\lambda_{0}$ is a bifurcation point, then the derivative with respect to $x$ of $\mathcal{F}$ in $(\lambda_{0},0)$ is not invertible. In particular, if \eqref{eq-intro-biforcation} is of the form
\begin{equation}\label{eq-intro-biforcation-2}
x-\lambda \mathcal{G}(x)=0, \quad (\lambda,x)\in \mathbb{R}\times X,
\end{equation}
where $\mathcal{G}$ is Fr\'echet differentiable and completely continuous, then every bifurcation point  is a characteristic value of the linearized problem $x=\lambda\mathrm{D}\mathcal{G}(0) x$, i.e., $1/\lambda$ is an eigenvalue of the linear compact operator $\mathrm{D}\mathcal{G}(0)$.

A seminal result in bifurcation theory, proposed by M.~A.~Krasnosel'ski\u{\i} in \cite{Kr-64} in the sixties, states that if $\lambda_0$ is a characteristic value of odd algebraic multiplicity of $\mathrm{D}\mathcal{G}(0)$, then $\lambda_0$ is a bifurcation point for \eqref{eq-intro-biforcation-2}.
In 1971, P.~Rabinowitz extended Krasnosel'ski\u{\i}'s result in \cite{Ra-71}, giving a global version of it. More precisely, given $\lambda_0$ as above, then $(\lambda_{0},0)$ belongs to the closure of a connected set $\Gamma$ of nontrivial solutions, which either is unbounded or contains a point $(\lambda_{1},0)$ with $\lambda_{1}\neq \lambda_{0}$.
These two fundamental results have initiated a line of research that has been extensively pursued in the subsequent years. Among a vast literature, the reader can see e.g.~\cite{ChHa-82,De-85,Fe-08,Ki-04,MaWa-05}.

Around the same time as the aforementioned work by Rabinowitz, A.~Ambrosetti and G.~Prodi published a milestone monograph in nonlinear analysis (see \cite{PrAm-73} for the original Italian edition, and \cite{AmPr-93} for the 1993 English edition). In the chapter concerning the bifurcation, the authors devoted a section of special bifurcation problems, which they called ``atypical''. This section examines specific examples of nonlinear differential problems that can be formulated as nonlinear functional equations of the form 
\begin{equation}\label{AmPr-atypical}
Lx-N(\lambda,x)=0, \quad (\lambda,x)\in\mathbb{R}\times X,
\end{equation}
in which $L\colon X\to X$ is a linear operator and $N$ satisfies $N(0,x)=0$, for all $x\in X$. 
Notice that, if $x\in \ker L$, then $(0,x)$ is a solution of \eqref{AmPr-atypical}; in other words $\{0\}\times\ker L$ has the same role of the line $\{(\lambda,0) \colon \lambda\in\mathbb{R}\}$ in the classical bifurcation theory discussed above.
In the context of \eqref{AmPr-atypical}, the bifurcation problem consists in looking for the existence of a point $x_0\in \ker L$ such that any neighbourhood of $(0,x_0)$ in $\mathbb{R}\times X$ contains a solution $(\lambda,x)$ of \eqref{AmPr-atypical} with $\lambda\neq 0$. 

Keeping the use of the terminology ``atypical bifurcation'' (sometimes also called ``cobifurcation''), this kind of problems received important contributions by M.~Furi and M.~P.~Pera \cite{Fu-83,FuPe-83}; we also mention more recent papers \cite{BFMP-05,BeZe-17,CaPeSp-pp,CaPeSp-24,FuPe-97,FuPe-97umi,FuPeSp-10}, also dealing with differential problems on manifolds.

Returning to our periodic problem \eqref{mainpb}, the following terminology will henceforth be employed, in accordance with the context of atypical bifurcation theory. A solution pair of \eqref{mainpb} of the type $(0,x)$ with $x$ constant is called \textit{trivial}. Clearly, a trivial solution pair occurs if and only if $f_0(x,0)=0$.
Given the Banach space $\mathcal{C}^1(\mathopen{[}0,T\mathclose{]},\mathbb{R}^n)$, endowed with the usual norm $\Vert \cdot \Vert_{\mathcal{C}^{1}}$, we denote by $\mathcal{C}^1_T$ the closed subspace of maps $x$ satisfying $x(0)=x(T)$ and $x'(0)=x'(T)$.
We say that a  trivial solution pair $(0,\bar{x})$ is a \textit{bifurcation point} for problem \eqref{mainpb} if any neighbourhood  of $(0,\bar{x})$ in $\mathbb{R} \times \mathcal{C}^1_T$ contains nontrivial solution pairs, that is, pairs of the type $(\lambda,x)$ with $\lambda\neq0$ or $(0,x)$ with $x$ nonconstant.

Our main result is the following.

\begin{theorem}\label{genbifthm}
Assume that \ref{hp-H1}, \ref{hp-H2} and \ref{hp-H3} hold true. Let $\Omega$ be an open subset of $\mathbb{R} \times  \mathcal{C}^1_T$ and suppose that the Brouwer degree
\begin{equation}
\label{gradononzeroo}
\mathrm{deg}_{\mathrm{B}}(f_0(\cdot,0),\widetilde\Omega_0,0)\neq 0,
\end{equation}
where $\widetilde\Omega_0\coloneqq\{x\in \mathbb{R}^n \colon (0,x)\in \Omega\}$ (with the natural identification between $\mathbb{R}^n$ and the subspace of constant maps in $\mathcal{C}^1_T$).
Then, there exists a connected set $\Gamma \subseteq \Omega$ of nontrivial solution pairs of \eqref{mainpb}, such that the closure of $\Gamma$ in $\Omega$ is not compact and intersects $\{0\} \times \widetilde\Omega_0$. 
\end{theorem}

The proof of Theorem~\ref{genbifthm} will be provided by a topological approach, specifically based on the coincidence degree and, in particular, its homotopy invariance property. 
It is important to emphasise that variational methods are not suitable in this context, where the second-order problem \eqref{mainpb} depends on the first derivative $x'$. Consequently, topological methods and degree conditions are typically employed to prove the emergence of solution branches. For results of this type in different frameworks, see for instance \cite{BFMP-05,FuPe-97}.

\bigskip

The paper is organised as follows. Section~\ref{section-2} deals with some topological preliminaries: the definition and the basic properties of the coincidence degree are recalled. In Section~\ref{section-3}, we first introduce an abstract problem in Banach spaces, equivalent to \eqref{mainpb}, and then we propose a finite-dimensional reduction theorem which will be crucial in the proof of Theorem~\ref{genbifthm}, which is given in Section~\ref{section-4} together with some corollaries. 
In Section~\ref{section-5}, we show some applications of our bifurcation result in the context of Li\'enard-type equations. We refer the reader to the classical contributions \cite{Ha-60,Kn-71,Ma-72jmaa,Re-75}, dealing with the linear second-order operator, and to the more recent ones \cite{BeMa-07,FeZa-21,Ma-13}, concerning nonlinear differential operators and which inspired our examples.
The paper ends with Appendix~\ref{appendix-A}, where we discuss some topological questions arisen along the paper and which can be of independent interest.

\section{Preliminaries on degree theory}\label{section-2}

For the reader's convenience, we first summarize the main facts concerning Mawhin's coincidence degree theory; for a comprehensive exposition, see \cite{GaMa-77,Ma-72,Ma-79,Ma-93}. Next, we recall a property of the Brouwer degree which will be crucial in the proof of the main theorem.

\subsection{The coincidence degree}\label{section-2.1}

Let $X$ and $Y$ be real Banach spaces and let $\mathcal{L} \colon \mathrm{dom}(\mathcal{L})\subseteq X\to Y$ be a (not necessarily bounded) Fredholm linear operator of index zero (that is, $\mathrm{dim}(\ker \mathcal{L})=\mathrm{codim}(\Im\mathcal{L})$ is finite and $\Im \mathcal{L}$ is closed in $Y$).
Let  $P\colon X\to X$ and $Q\colon Y\to Y$ be two continuous projections such that $\Im P=\ker \mathcal{L}$ and $\ker Q=\Im \mathcal{L}$. Since $\Im Q$ and $\ker \mathcal{L} $ have the same finite dimension, there exists an isomorphism $J\colon \Im Q\to \ker \mathcal{L}$. It is readily seen that $\mathcal{L}|_{\mathrm{dom}(\mathcal{L})\cap \ker P}\colon \mathrm{dom}(\mathcal{L})\cap \ker P\to \Im \mathcal{L}$ is bijective, thus it has an algebraic inverse $K_P$.

Let $W$ be an open subset of $X$. A continuous map $\mathcal{N}\colon W\to Y$ is called \textit{$\mathcal{L}$-completely continuous} if $\mathcal{N}(B)$ is bounded in $Y$ for each bounded $B\subseteq W$ and the map $K_P(\mathrm{Id}_Y-Q)\mathcal{N}\colon  W\to X$ is completely continuous, where $\mathrm{Id}_Y$ is the identity on $Y$.

Given an open subset $\mathcal{O}$ of $W$, the pair $(\mathcal{L} - \mathcal{N}, \mathcal{O})$ is said to be \textit{admissible} for the coincidence degree if the set of solutions in $\mathcal{O}\cap \dom( \mathcal{L})$  of 
\begin{equation}\label{functional equation}
\mathcal{L} u = \mathcal{N} (u)
\end{equation}
is compact. It is not difficult to see that \eqref{functional equation} is equivalent to 
\begin{equation}\label{def-Phi-degree}
u= \Phi(u),
\end{equation}
where $\Phi\colon W\to X$ is given by 
\begin{equation*}
\Phi(u) = Pu + JQ\mathcal{N}(u) + K_P(\mathrm{Id}_Y - Q)\mathcal{N}(u).
\end{equation*}
Notice that $\Phi$ is completely continuous since $\mathcal{N}$ is $\mathcal{L}$-completely continuous. 
In addition, $\{u\in \mathcal{O} \colon u= \Phi(u)\}$ is compact since the set of solutions of \eqref{functional equation} is assumed to be compact. Consequently, the Leray--Schauder degree $\mathrm{deg}_{\mathrm{LS}}(\mathrm{Id}_X-\Phi,\mathcal{O})$ is well defined. Therefore, the coincidence degree of $(\mathcal{L} - \mathcal{N}, \mathcal{O})$ is defined as the integer
\begin{equation*}
\mathrm{D}_{\mathcal{L}}(\mathcal{L} - \mathcal{N}, \mathcal{O})\coloneqq\mathrm{deg}_{\mathrm{LS}}(\mathrm{Id}_X-\Phi,\mathcal{O},0).
\end{equation*}

\begin{remark}\label{J determina il segno}
The definition of $\Phi$ depends, among other things, on $J$. Roughly speaking, $J$ determines a choice of orientations on $\ker \mathcal{L}$ and $\Im Q$, up to an inversion of both of them, in the sense that $J$ is implicitly assumed to be orientation preserving. If $J'$ were another isomorphism, one would obtain the same value of the degree, up to the sign.
\hfill$\lhd$
\end{remark}

The coincidence degree satisfies the classical properties in the topological degree theory. We list below some of them, with particular attention to those used in this paper.

\begin{proposition}
\label{degree properties}
The coincidence degree satisfies the following properties.

\begin{itemize}
\item Normalization. Given an admissible pair $(\mathcal{L} - \mathcal{N}, \mathcal{O})$,  assume that $\mathcal{L}$ is injective and $\mathcal{N}$ is constant and equal to $b$. If $b\in \mathcal{L} (\mathcal{O} \cap \dom(\mathcal{L}))$, then 
\begin{equation*}
|\mathrm{D}_{\mathcal{L}}(\mathcal{L} - b, \mathcal{O})| = 1.
\end{equation*}

\item Additivity. If $(\mathcal{L} - \mathcal{N}, \mathcal{O})$ is  admissible and $\mathcal{O}_1$, $\mathcal{O}_2$ are two disjoint open subsets of $\mathcal{O}$ such that $(\mathcal{L} - \mathcal{N})^{-1}(0)\cap \mathcal{O}  \subseteq \mathcal{O}_1 \cup \mathcal{O}_2$, then
\begin{equation*}
\mathrm{D}_{\mathcal{L}}(\mathcal{L} - \mathcal{N}, \mathcal{O}) = \mathrm{D}_{\mathcal{L}}(\mathcal{L} - \mathcal{N}, \mathcal{O}_1) +\mathrm{D}_{\mathcal{L}}(\mathcal{L} - \mathcal{N}, \mathcal{O}_2).
\end{equation*}

\item Excision. If $(\mathcal{L} - \mathcal{N}, \mathcal{O})$ is  admissible and $\mathcal{V}$ is an open subset of $\mathcal{O}$ containing $(\mathcal{L} - \mathcal{N})^{-1}(0)\cap \mathcal{O}$, then
\begin{equation*}
\mathrm{D}_{\mathcal{L}}(\mathcal{L} - \mathcal{N}, \mathcal{O}) = \mathrm{D}_{\mathcal{L}}(\mathcal{L} - \mathcal{N}, \mathcal{V}) .
\end{equation*}

\item Homotopy invariance. Given an open subset $\Omega$ of $\mathopen{[}0,1\mathclose{]}\times X$, let $\mathcal{H}\colon \Omega \cap (\mathopen{[}0,1\mathclose{]}\times \dom(\mathcal{L})) \to Y$ be a homotopy of the form
\begin{equation*}
\mathcal{H} (\lambda,u)=\mathcal{L} u - \widetilde{\mathcal{N}} (\lambda,u),
\end{equation*}
and assume that $\widetilde{\mathcal{N}} \colon \Omega \to Y$ is $\mathcal{L}$-completely continuous (i.e., $\widetilde{\mathcal{N}}(B)$ is bounded and $K_P(\mathrm{Id}_Y-Q)\widetilde{\mathcal{N}}(B)$ is relatively compact for every bounded set $B\subseteq \Omega$).
If the set 
\begin{equation*}
\bigl{\{}u\in X\colon \mathcal{H}(\lambda,u)=0 \text{ for some $\lambda\in \mathopen{[}0,1\mathclose{]}$}\bigr{\}}
\end{equation*}
is compact, then $\mathrm{D}_{\mathcal{L}}(\mathcal{H}_\lambda,\Omega_\lambda)$ is well defined and is constant in $\mathopen{[}0,1\mathclose{]}$, where $\mathcal{H}_\lambda\coloneqq \mathcal{H}(\lambda,\cdot)$ and $\Omega_\lambda\coloneqq\{u\in X \colon (\lambda,u)\in \Omega\}$, for every $\lambda\in \mathopen{[}0,1\mathclose{]}$.

\item Existence. If $(\mathcal{L} - \mathcal{N}, \mathcal{O})$ is admissible and
\begin{equation*}
\mathrm{D}_{\mathcal{L}}(\mathcal{L} - \mathcal{N}, \mathcal  O)\neq 0,
\end{equation*}
then the equation $\mathcal{L} u = \mathcal{N} (u)$ has at least one solution in $\mathcal  O$.
\end{itemize}
\end{proposition}

\subsection{A composition property for the Brouwer degree}\label{section-2.2}

The construction and the properties of the Brouwer degree are very well known (cf.~\cite{FoGa-95,Ll-78}). However, we find useful to recall a particular case of the so-called \textit{composition property} of the degree. 
This will help to better understand a specific step in the final part of the proof of Proposition~\ref{prop_reduction_property} given in Section~\ref{section-3}.
For a general (and more complicated) statement, the reader can see e.g.~\cite[Theorem~2.10]{FoGa-95}.

\begin{proposition}\label{prop_composition_property}
Let $\mathcal{O}_1$ and $\mathcal{O}_2$ be two open subsets of $\mathbb{R}^n$ and $f\colon\mathcal{O}_1 \times \mathcal{O}_2 \to \mathbb{R}^{2n}$ a continuous function. Assume that $f$ is of the form $f(x_1,x_2)=(f_1(x_2),f_2(x_1))$. Suppose that the set of solutions of $f(x_1,x_2)=0$ is compact. In addition, denote by $g\colon\mathcal{O}_2 \times \mathcal{O}_1 \to \mathbb{R}^{2n}$ the map defined as $g(x_2,x_1)=(f_1(x_2),f_2(x_1))$. Then,
\begin{align*}
\mathrm{deg}_{\mathrm{B}} (f, \mathcal{O}_1 \times \mathcal{O}_2,0)
&=(-1)^n \, \mathrm{deg}_{\mathrm{B}}(g,\mathcal{O}_2 \times \mathcal{O}_1,0)
\\
&= (-1)^n \, \mathrm{deg}_{\mathrm{B}}(f_1,\mathcal{O}_2,0)
\, \mathrm{deg}_{\mathrm{B}}(f_2, \mathcal{O}_1,0).
\end{align*}
\end{proposition}

\section{A functional formulation of problem~\eqref{mainpb}}
\label{section-3}

To study the periodic problem \eqref{mainpb}, under conditions \ref{hp-H1}, \ref{hp-H2} and \ref{hp-H3}, we proceed in the spirit of \cite{FeZa-17} and consider the following equivalent first-order problem 
\begin{equation}\label{mainpbfo}
\begin{cases}
\, x_1'=\phi^{-1}(x_2),
\\
\, x_2'=F(\lambda,t,x_1,\phi^{-1}(x_2)),
\\
\, x_1(0)=x_1(T),\quad x_2(0)=x_2(T).
\end{cases}
\end{equation}

Consider the Banach space $\mathcal{C}(\mathopen{[}0,T\mathclose{]},\mathbb{R}^n)$ endowed with the usual norm $\|\cdot\|_{\infty}$ and denote by $\mathcal{AC}_T$ the subspace of the absolutely continuous maps $x$ satisfying $x(0)=x(T)$. 
Given the Banach space $L^1\coloneqq L^1(\mathopen{[}0,T\mathclose{]},\mathbb{R}^n)$ with the usual norm $\|\cdot\|_{L^{1}}$, problem \eqref{mainpbfo} can be written as a nonlinear functional equation parametrized by $\lambda$:
\begin{equation}
\label{nonlinearcoincidenceequation}
\mathcal{L}(x_1,x_2)=\mathcal{N}(\lambda,x_1,x_2),
\end{equation}
in which $\mathcal{L}\colon \mathcal{AC}_T\times \mathcal{AC}_T \to L^1\times L^1$ is the linear operator 
\begin{equation*}
\mathcal{L}(x_1,x_2)=(x_1',x_2'),
\end{equation*}
and $\mathcal{N}\colon\mathbb{R} \times \mathcal{C}(\mathopen{[}0,T\mathclose{]},\mathbb{R}^n)\times \mathcal{C}(\mathopen{[}0,T\mathclose{]},\mathbb{R}^n) \to L^1\times L^1$ is the Nemitskii nonlinear operator
\begin{equation*}
\mathcal{N}(\lambda,x_1,x_2)(t)=\left(\phi^{-1}(x_2(t)),F(\lambda,t,x_1(t),\phi^{-1}(x_2(t)))\right).
\end{equation*}
Notice that $\mathcal{N}$ is $\mathcal{L}$-completely continuous (cf.~\cite{Ma-79}).

To simplify the notation, we will write $\dom(\mathcal{L})$ in place of $\mathcal{AC}_T\times \mathcal{AC}_T$.
It is known that $\mathcal{L}$ is an unbounded Fredholm linear operator of index zero. Its kernel  is the $2n$-dimensional subspace of constant maps and its image is made of pairs of functions with zero mean value. 

Recalling the previous section, consider the following projections:
\begin{align*}
&P\colon \dom(\mathcal{L})\to \dom(\mathcal{L}), \quad P(x_1,x_2)=\dfrac{1}{T} \biggl{(}\int_{0}^T x_1(s)\,\mathrm{d}s, \; \int_{0}^T x_2(s)\,\mathrm{d}s\biggr{)},
\\
&Q\colon L^1\times L^1\to L^1\times L^1, \quad Q(y_1,y_2)=\dfrac{1}{T} \biggl{(}\int_{0}^T y_1(s)\,\mathrm{d}s, \; \int_{0}^T y_2(s)\,\mathrm{d}s\biggr{)}.
\end{align*}  
It is immediate to check that $\Im P=\ker \mathcal{L}$ and $\ker Q=\Im \mathcal{L}$.

Finally, we denote by $J\colon\Im Q\to \ker \mathcal{L}$ the isomorphism $J(y_1,y_2)=(y_1,y_2)$. 
It associates to any pair of constant maps $(y_1,y_2)\in L^1\times L^1$ the same pair $(y_1,y_2)$, though considered as an element of $\dom(\mathcal{L})$.

Let us now tackle problem \eqref{nonlinearcoincidenceequation} in the particular case when $\lambda=0$. We set 
\begin{equation*}
\mathcal{M}_{0}(x_1,x_2)\coloneqq\mathcal{N}(0,x_1,x_2)=\left(\phi^{-1}(x_2),f_0(x_1,\phi^{-1}(x_2))\right).
\end{equation*} 

The following proposition is a finite-dimensional reduction property for the coincidence degree and it will be important in the obtention of our bifurcation result. We give the proof for a sake of completeness.

\begin{proposition}\label{prop_reduction_property}
Let $\mathcal{O}$ be an open subset of $\mathcal{C}(\mathopen{[}0,T\mathclose{]},\mathbb{R}^n)\times \mathcal{C}(\mathopen{[}0,T\mathclose{]},\mathbb{R}^n)$ such that 
\begin{enumerate}[leftmargin=26pt,labelsep=6pt,label=\textup{$(h)$}]
\item the set $\mathcal{S}_{0}\coloneqq\{(x_1,x_2)\in\mathcal{O}\cap\dom \mathcal{L} \colon \mathcal{L}(x_1,x_2)=\mathcal{M}_{0}(x_1,x_2)\}$ is compact.
\label{hp-h-prop}
\end{enumerate}
Then,
\begin{equation*}
\vert \mathrm{D}_{\mathcal{L}}(\mathcal{L}-\mathcal{M}_{0},\mathcal{O} )\vert =\vert \mathrm{deg}_{\mathrm{B}}(f_0(\cdot,0),\mathcal{O}_{1,0}\cap\mathbb{R}^n,0)\vert,
\end{equation*}
where  $\mathcal{O}_{1,0}\coloneqq\{x_1\in \mathcal{C}(\mathopen{[}0,T\mathclose{]},\mathbb{R}^n)\colon (x_1,0)\in \mathcal{O}\}$.
\end{proposition}

\begin{proof}
The first step is based on the continuation theorem by Capietto, Mawhin, and Zanolin \cite[Theorem~1]{CMZ-92} in the version of \cite[Theorem~A.3]{FeZa-17}.
Accordingly, by hypothesis \ref{hp-h-prop}, we obtain that 
\begin{equation*}
\mathrm{D}_{\mathcal{L}}(\mathcal{L}-\mathcal M,\mathcal{O})
=\mathrm{deg}_{\mathrm{B}}(\mathfrak{f}_{0},\mathcal{O}\cap \ker \mathcal{L},0),
\end{equation*}
where
\begin{equation*}
\mathfrak{f}_{0}(x_1,x_2)= \left(\phi^{-1}(x_2),f_0(x_1,\phi^{-1}(x_2))\right), \quad (x_1,x_2)\in\mathbb{R}^{2n}.
\end{equation*}

Let us now show that 
\begin{align*}
&\mathrm{deg}_{\mathrm{B}}(\mathfrak{f}_{0},\mathcal{O}\cap \ker \mathcal{L},0)=
\\
&=(-1)^n \, \mathrm{deg}_{\mathrm{B}}(\phi^{-1},\mathcal{O}_2\cap\ker \mathcal{L}_2,0) \, \mathrm{deg}_{\mathrm{B}}(f_0(\cdot,0),\mathcal{O}_{1,0}\cap\ker \mathcal{L}_1,0),
\end{align*}
where $\mathcal{L}(x_1,x_2)=(\mathcal{L}_1x_1,\mathcal{L}_2x_2)$ and $\mathcal{O}_2$ is defined as
\begin{equation*}
\mathcal{O}_2\coloneqq \bigl{\{} x_2\in \dom(\mathcal{L}_2)\colon (x_1,x_2)\in \mathcal{O},\; \text{for some $x_1$} \bigr{\}}.
\end{equation*}
We introduce the homotopy $H\colon \mathopen{[}0,1\mathclose{]}\times (\mathcal{O}\cap\ker \mathcal{L})\to \ker \mathcal{L} $, given by
\begin{equation*}
H(\vartheta,x_1,x_2)=
\left( \phi^{-1}(x_2), f_{0}(x_1,\vartheta x_2)\right).
\end{equation*}
Observe that, for every $(\vartheta, x_1,x_2)\in\mathopen{[}0,1\mathclose{]}\times (\mathcal{O}\cap\ker \mathcal{L})$,
\begin{equation*}
H(\vartheta,x_1,x_2) = 0
\quad \Longrightarrow \quad x_2=0,
\end{equation*}
by hypothesis \ref{hp-H1}.
Therefore, the set of solutions of $H(\vartheta,x_1,x_2)=0$ coincides with $\mathopen{[}0,1\mathclose{]}\times (\mathcal{S}_{0}\cap \ker \mathcal{L})$, which is compact by hypothesis \ref{hp-h-prop}. Applying the homotopy invariance property of the Brouwer degree, we have
\begin{align*}
\mathrm{deg}_{\mathrm{B}}(\mathfrak{f}_{0},\mathcal{O}\cap \ker \mathcal{L},0)
&=\mathrm{deg}_{\mathrm{B}}(H(1,\cdot,\cdot),\mathcal{O}\cap \ker \mathcal{L},0)
\\
&=\mathrm{deg}_{\mathrm{B}}(H(0,\cdot,\cdot),\mathcal{O}\cap \ker \mathcal{L},0).
\end{align*}
It is straightforward to see that the sets
\begin{align*}
&\bigl{\{} (x_1,x_2)\in \mathcal{O}\cap \ker \mathcal{L}\colon H(0,x_1,x_2)=0 \bigr{\}} \quad \text{and}\\ 
&\bigl{\{}(x_1,x_2)\in  (\mathcal{O}_{1,0}\cap \ker \mathcal{L}_1)\times(\mathcal{O}_2\cap \ker \mathcal{L}_2)\colon H(0,x_1,x_2)=0\bigr{\}}
\end{align*}
coincide. Therefore, by the excision property of the Brouwer degree, one has 
\begin{equation*}
\mathrm{deg}_{\mathrm{B}}(H(0,\cdot,\cdot),\mathcal{O}\cap \ker \mathcal{L},0)=\mathrm{deg}_{\mathrm{B}}(H(0,\cdot,\cdot),(\mathcal{O}_{1,0}\cap \ker \mathcal{L}_1)\times(\mathcal{O}_2\cap \ker \mathcal{L}_2),0),
\end{equation*}
and, recalling Proposition~\ref{prop_composition_property},
\begin{align*}
&\mathrm{deg}_{\mathrm{B}}(H(0,\cdot,\cdot),(\mathcal{O}_{1,0}\cap \ker \mathcal{L}_1)\times(\mathcal{O}_2\cap \ker \mathcal{L}_2),0)\\
&=(-1)^n \, \mathrm{deg}_{\mathrm{B}}((-\phi^{-1},-f_0(\cdot,0)),(\mathcal{O}_2\cap \ker \mathcal{L}_2)\times (\mathcal{O}_{1,0}\cap \ker \mathcal{L}_1),0)\\
&=(-1)^n \, \mathrm{deg}_{\mathrm{B}}(-\phi^{-1},\mathcal{O}_2\cap \ker \mathcal{L}_2,0)
\, \mathrm{deg}_{\mathrm{B}}(-f_0(\cdot,0),\mathcal{O}_{1,0}\cap \ker \mathcal{L}_1,0).
\end{align*}

A classical result in Brouwer degree theory says that $\mathrm{deg}_{\mathrm{B}}(\psi,U,0)=\pm1$ if $U$ is an open subset of $\mathbb{R}^n$ and $\psi\colon U\to \mathbb{R}^n$ is continuous, injective with $0\in \psi(U)$ (see e.g.~\cite[Theorem~7.4.5]{DiMa-21} or, more generally, \cite[Proposition~II.17]{Ma-79}). 
By the multiplicativity property (cf.~\cite[Theorem~2.10]{FoGa-95}), we finally deduce that
\begin{equation*}
\mathrm{deg}_{\mathrm{B}}(-f_0(\cdot,0),\mathcal{O}_{1,0}\cap \ker \mathcal{L}_1,0)=(-1)^{n}\mathrm{deg}_{\mathrm{B}}(f_0(\cdot,0),\mathcal{O}_{1,0}\cap \ker \mathcal{L}_1,0)
\end{equation*}
and the proof is complete.
\end{proof}

\section{Proof of Theorem~\ref{genbifthm} and consequences}\label{section-4}

A crucial role in the proof is played by the following  technical result that is a Whyburn-type lemma.
For the proof we refer the reader to  \cite[Lemma~1.4]{FuPe-93}. 

\begin{lemma}\label{whyburn}
Let $Y_0$ be a compact subset of a locally compact metric space $Y$. Assume that any compact subset of $Y$ containing $Y_0$ has nonempty boundary. Then, $Y \backslash Y_0$ contains a not relatively compact connected component whose closure in $Y$ intersects $Y_0$.
\end{lemma}

\begin{proof}[Proof of Theorem~\ref{genbifthm}]
Let us introduce the sets
\begin{align*}
A&\coloneqq\bigl\{(\lambda,x_{1},x_{2}) \in \mathbb{R} \times \mathcal{AC}_T\times \mathcal{AC}_T\colon \mathcal{L} (x_{1},x_{2}) =\mathcal{N}(\lambda,x_{1},x_{2}) \bigr\},
\\
\Sigma &\coloneqq \bigl\{(\lambda,x_{1},x_{2}) \in \Omega\times \mathcal{AC}_T\colon \mathcal{L} (x_{1},x_{2}) =\mathcal{N}(\lambda,x_{1},x_{2}) \bigr\},
\\
\widetilde{\Sigma}_0 &\coloneqq \bigl\{(x_{1},x_{2})\in \widetilde\Omega_0\times \mathcal{AC}_T \colon \mathcal{L} (x_{1},x_{2})=  \mathcal{N}(0,x_{1},x_{2}) \bigr\}.
\end{align*}
We point out that we will consider $A$ and $\Sigma$ endowed with the topology of $\mathbb {R}\times \mathcal{C}(\mathopen{[}0,T\mathclose{]},\mathbb{R}^n)\times \mathcal{C}(\mathopen{[}0,T\mathclose{]},\mathbb{R}^n)$, while $\widetilde{\Sigma}_0$ is endowed with the topology of $ \mathcal{C}(\mathopen{[}0,T\mathclose{]},\mathbb{R}^n)\times \mathcal{C}(\mathopen{[}0,T\mathclose{]},\mathbb{R}^n)$.

We aim to apply Lemma~\ref{whyburn} to the pair $(\Sigma,\{0\}\times \widetilde{\Sigma}_0)$ in place of $(Y, Y_0)$.
First, we notice that 
\begin{equation*}
\widetilde{\Sigma}_0=\bigl{\{} (x_{1},0)\in\widetilde\Omega_0\times\{0\} \colon f_{0}(x_{1},0)=0 \bigr{\}}
\end{equation*}
and thus it is compact since the degree in \eqref{gradononzeroo} is well defined. 
Second, we need to show that $\Sigma$ is locally compact. 

Preliminarily, we verify that $A$ is locally compact. To see this, recalling the equivalence of equations \eqref{functional equation} and \eqref{def-Phi-degree} in a general setting, observe that, in an analogous way, the equation
$\mathcal{L} (x_1,x_2)= \mathcal{N} (\lambda,x_1,x_2)$ is equivalent to $(x_1,x_2)= \Phi(\lambda,x_1,x_2)$, in which $\Phi\colon \mathbb{R} \times \mathcal{C}(\mathopen{[}0,T\mathclose{]},\mathbb{R}^n)\times \mathcal{C}(\mathopen{[}0,T\mathclose{]},\mathbb{R}^n) \to \mathcal{C}(\mathopen{[}0,T\mathclose{]},\mathbb{R}^n)\times \mathcal{C}(\mathopen{[}0,T\mathclose{]},\mathbb{R}^n)$ is given by 
\begin{equation}
\label{Phi  completely continuous}
\Phi(\lambda,x_1,x_2) = P(x_1,x_2) + JQ\mathcal{N}(\lambda,x_1,x_2) + K_P(\mathrm{Id}_{L^1\times L^1} - Q)\mathcal{N}(\lambda,x_1,x_2),
\end{equation}
where $P$, $Q$ and $J$ have been defined in Section \ref{section-3}, and $K_P$ is the algebraic inverse of $\mathcal{L}|_{\mathrm{dom}(\mathcal{L})\cap \ker P}$.
Since $\mathcal{N}$ is $\mathcal{L}$-completely continuous, then $\Phi$ is completely continuous. 
The map $(\lambda,x_1,x_2) \mapsto (x_1,x_2)$ is proper on closed and bounded subsets of $\mathbb{R} \times \mathcal{C}(\mathopen{[}0,T\mathclose{]},\mathbb{R}^n)\times \mathcal{C}(\mathopen{[}0,T\mathclose{]},\mathbb{R}^n)$, and thus so is the map
\begin{equation*}
(\lambda,x_1,x_2) \mapsto (x_1,x_2) - \Phi(\lambda,x_1,x_2),
\end{equation*}
because it is a completely continuous perturbation of a map that, as said before, is proper on closed and bounded sets (see Appendix~\ref{appendix-A} for the details). This implies that $A$ is compact if intersected with closed and bounded sets of $\mathbb{R} \times \mathcal{C}(\mathopen{[}0,T\mathclose{]},\mathbb{R}^n)\times \mathcal{C}(\mathopen{[}0,T\mathclose{]},\mathbb{R}^n)$, and, consequently, $A$ is locally compact.

Since an open subset of a locally compact metric space is locally compact, to obtain the local compactness of $\Sigma$ it is sufficient to prove that $\Sigma$ is open in $A$. 
To this purpose, fix $(\bar{\lambda},\bar{x}_1, \bar{x}_2)\in \Sigma$. As $(\bar{\lambda},\bar{x}_1)\in \Omega$, which is open in $\mathbb{R} \times  \mathcal{C}^1_T$, let $\varepsilon>0$ be such that $(\bar{\lambda}-\varepsilon,\bar{\lambda}+\varepsilon)\times B^1(\bar{x}_1,\varepsilon)\subseteq \Omega$, where $B^1(\bar{x}_1,\varepsilon)$ denotes the ball in $\mathcal{C}^1_T$ with center $\bar{x}_1$ and radius $\varepsilon$. 

Consider the nonlinear operator $\Psi \colon \mathcal{C}(\mathopen{[}0,T\mathclose{]},\mathbb{R}^n) \to \mathcal{C}(\mathopen{[}0,T\mathclose{]},\mathbb{R}^n)$ defined by
\begin{equation*}
\Psi(u)(t)=\phi^{-1}(u(t)), \quad t\in \mathopen{[}0,T\mathclose{]}.
\end{equation*}
Observe that $\Psi$ is a homeomorphism since so is $\phi$. Moreover, $\Psi(\bar{x}_2)=\bar{x}_1'$.
Consequently, there exists a ball  $B^0(\bar{x}_2,\delta)$ in $\mathcal{C}(\mathopen{[}0,T\mathclose{]},\mathbb{R}^n)$ such that 
\begin{equation*}
\Psi (B^0(\bar{x}_2,\delta))\subseteq B^0(\bar{x}_1',\varepsilon/2).
\end{equation*}
Denote by $U$ the open subset of $\mathbb{R} \times \mathcal{C}(\mathopen{[}0,T\mathclose{]},\mathbb{R}^n)\times \mathcal{C}(\mathopen{[}0,T\mathclose{]},\mathbb{R}^n)$ given by
\begin{equation*}
U\coloneqq(\bar{\lambda}-\varepsilon,\bar{\lambda}+\varepsilon)\times B^0(\bar{x}_1,\varepsilon/2) \times B^0(\bar{x}_2,\delta)
\end{equation*}
and call $U_A\coloneqq U\cap A$, which is an open neighbourhood of $(\bar{\lambda},\bar{x}_1,\bar{x}_2)$ in $A$. Observe that every $(\lambda,  x_1,  x_2)\in U_A$ satisfies $x_1\in\mathcal{C}^1(\mathbb{R})$ and $\Psi(x_2)=x_1'$ (since $(\lambda,  x_1,  x_2)\in A$), and moreover $\| x_1' - \bar{x}_1'\| _\infty<\varepsilon/2$ (since $\| x_2- \bar{x}_2 \| _\infty<\delta$). 
This implies that $\|x_1-\bar{x}_1 \| _{\mathcal{C}^{1}} = \|x_{1}-\bar{x}_1\|_{\infty} +\|x_{1}'-\bar{x}_1'\|_{\infty} <\varepsilon$ and thus $( \lambda,  x_1)\in \Omega$, that is, $( \lambda,  x_1,  x_2)\in \Sigma$. Therefore, $U_A$ is actually contained in $\Sigma$ and this proves that $\Sigma$ is open in $A$. We conclude that $\Sigma$ is locally compact.

Next, to apply Lemma~\ref{whyburn}, we need to show that any compact subset of $\Sigma$ containing $\{0\}\times\widetilde{\Sigma}_0$ has nonempty boundary.
We proceed by contradiction and suppose that there exists a compact set $C \subseteq \Sigma$ that contains $\{0\} \times  \widetilde{\Sigma}_0$ and has empty boundary in $\Sigma$. 
Hence, $C$ is open in $\Sigma$.
This implies the existence of an open (and bounded) subset $W$ of $\mathbb{R}\times \mathcal{C}(\mathopen{[}0,T\mathclose{]},\mathbb{R}^n)\times \mathcal{C}(\mathopen{[}0,T\mathclose{]},\mathbb{R}^n)$ containing $C$ and such that $W \cap \Sigma = C$.
Furthermore, without loss of generality, we may assume that $W \cap A = C$. 
Indeed, since $\Sigma$ is open in $A$, there exists an open set $W'\subseteq \mathbb {R}\times \mathcal{C}(\mathopen{[}0,T\mathclose{]},\mathbb{R}^n)\times \mathcal{C}(\mathopen{[}0,T\mathclose{]},\mathbb{R}^n)$ such that $W'\cap A=\Sigma$. 
Then, since $C$ has empty boundary in $\Sigma$, we can define $W$ as a union of open balls with centers at points of $C$, contained in $W'$ and disjoint from $\Sigma\setminus C$. Since $W\subseteq W'$, then $W \cap A = C$.

As a consequence, the degree  $\mathrm{D}_{\mathcal{L}}(\mathcal{L}-\mathcal{N}(\lambda,\cdot,\cdot),W_\lambda,0)$ is well defined, for any $\lambda\in \mathbb{R}$, and the homotopy invariance property implies that it does not depend on $\lambda \in \mathbb{R}$. 
Here,  $W_\lambda$ stands  for the set of those $(x_{1},x_{2})$ such that $(\lambda,x_{1},x_{2})\in W$. We use an analogous notation for the sections of $\Sigma$ and $C$.

Observe that $\widetilde{\Sigma}_0$ and, consequently, $C_{0}$ are not empty by hypothesis \eqref{gradononzeroo}. In addition, by the compactness of $C$, the set $\{\lambda\in \mathbb{R} \colon C_\lambda\neq \emptyset\}$ has minimum $\lambda_{-}$ and maximum $\lambda_{+}$ with $\lambda_{-}\leq0\leq\lambda_{+}$. Since $W$ is open, there exists $\lambda' \in (-\infty,\lambda_{-})\cup(\lambda_{+},+\infty)$ such that $W_{\lambda'}\neq \emptyset$ and, of course, $C_{\lambda'}=\emptyset$.
Recalling that $W_{\lambda'}\cap \Sigma_{\lambda'}= C_{\lambda'}=\emptyset$, by the existence property of the degree we have
\begin{equation*}
\mathrm{D}_{\mathcal{L}}(\mathcal{L}-\mathcal{N}(\lambda',\cdot,\cdot),W_{\lambda'},0)=0.
\end{equation*}
By the homotopy invariance property, it follows 
\begin{equation}
\label{gradozerolambdazero}
\mathrm{D}_{\mathcal{L}}(\mathcal{L}-\mathcal{N}(0,\cdot,\cdot),W_0,0)=0.
\end{equation}
Observing that $C_0$ is compact and $W_0 \cap \Sigma_0 = C_{0}$, hypothesis \ref{hp-h-prop} of Proposition~\ref{prop_reduction_property} (with $\mathcal{O}=W_{0}$) is satisfied and thus we deduce
\begin{equation}
\label{gradozerolambdazeroreduction}
\vert \mathrm{D}_{\mathcal{L}}(\mathcal{L}-\mathcal{N}(0,\cdot,\cdot),W_0,0)\vert =\vert \mathrm{deg}_{\mathrm{B}}(f_0(\cdot,0),\widetilde W_{0},0)\vert,
\end{equation}
where $\widetilde W_{0} \coloneqq \{ x_1\in \mathbb{R}^n \colon (x_{1},0)\in W_0\}$.  
By the construction of $W$, one can observe that $\widetilde \Omega_0$ and $\widetilde W_{0}$ are two open neighbourhoods in $\mathbb{R}^n$ of the set $\{x_1\in \mathbb{R}^n \colon (x_{1},0)\in\widetilde{\Sigma}_0\}$. 
Moreover, recalling the definition of $\widetilde\Sigma_0$ and the fact that $W\cap A=\Sigma$, we have
\begin{align*}
&\{(x_1,0)\in \widetilde \Omega_0 \times\{0\} \colon \mathcal{L} (x_{1},0)=  \mathcal{N}(0,x_{1},0)\}=
\\  
&= \{(x_1,0)\in \widetilde W_{0} \times\{0\} \colon \mathcal{L} (x_{1},0)=  \mathcal{N}(0,x_{1},0)\} 
= \widetilde\Sigma_0.
\end{align*}
Therefore, we can exploit the excision property of the degree together with \eqref{gradononzeroo} obtaining
\begin{equation*}
\mathrm{deg}_{\mathrm{B}}(f_0(\cdot,0),\widetilde W_0,0)=
\mathrm{deg}_{\mathrm{B}}(f_0(\cdot,0),\widetilde\Omega_0,0)
\neq 0.
\end{equation*}
This is in contradiction with \eqref{gradozerolambdazero} via \eqref{gradozerolambdazeroreduction}. 
Consequently, the assumptions of Lemma~\ref{whyburn} are satisfied and hence there exists a connected set $\Xi\subseteq \Sigma\setminus (\{0\}\times\widetilde{\Sigma}_0)$ whose closure in $\Sigma$ is not compact and intersect $\{0\}\times\widetilde{\Sigma}_0$. 

To conclude the proof, we define
\begin{equation}
\label{definition of Gamma}\Gamma \coloneqq \bigl\{ (\lambda,x)\in \mathbb{R}\times \mathcal{C}^{1}_{T} \colon (\lambda,x,\phi(x'))\in \Xi \bigr\}.
\end{equation}
It is immediate to see that $\Gamma$ is contained in $\Omega$, by the definition of $\Sigma$. We claim in addition that 
\begin{itemize}
\item [$(i)$] $\Gamma$ is connected (in the topology of $\mathbb{R}\times \mathcal{C}^{1}_{T}$); 
\item [$(ii)$] the closure $\overline\Gamma$ of $\Gamma$ in $\Omega$ intersects $\{0\} \times \widetilde\Omega_0$;
\item [$(iii)$] $\overline\Gamma$  is not compact.
\end{itemize}

To prove the above facts, define the map $ \xi\colon \Sigma\to \mathbb R \times \mathcal{C}^1_T$ as
\begin{equation}\label{def-map-xi}
 \xi (\lambda, x_1,x_2) =(\lambda,x_1).
\end{equation}
Observe that $ \xi$ is continuous. Indeed, consider a sequence $(\lambda_n,  x_{1,n},  x_{2,n})_n$ in  $\Sigma$ converging to a point $(\bar\lambda,  \bar x_1,  \bar x_2)\in \Sigma$. Since every element of the sequence, as well as its limit, is a solution of $\mathcal{L} (x_{1},x_{2}) =\mathcal{N}(\lambda,x_{1},x_{2})$, then $( x_{1,n})'=\phi^{-1}( x_{2,n})$ and  $\bar x_1'=\phi^{-1}(\bar x_2)$. By the continuity of the map $\Psi$ previously defined, the convergence to zero of $\Vert x_{2,n} -  \bar x_2 \Vert_\infty $ implies the same for $\Vert ( x_{1,n})' -  \bar x_1' \Vert_\infty $. Therefore, $(x_{1,n})_n$ converges to $\bar x_1$ in the $\mathcal{C}^1$-norm. We can conclude that $ \xi$ is continuous.

Concerning property $(i)$, observe that $\xi(\Xi)=\Gamma$ and this proves that this set is connected since so is $\Xi$.
As for property $(ii)$, take a sequence $(\lambda_n,  x_{1,n},  x_{2,n})_n$ in $\Xi$ that converges to an element $(0,\bar x_1,0)\in \{0\}\times \widetilde{\Sigma}_0$. 
By the continuity of $\xi$, the sequence $(\lambda_n,  x_{1,n})_n$, which is contained in $\Gamma$, converges (in the topology of $\mathbb R \times \mathcal{C}^1_T$) to $(0,\bar x_1)\in\{0\} \times \widetilde\Omega_0$, namely the closure $\overline\Gamma$ of $\Gamma$ in $\Omega$ intersects $\{0\} \times \widetilde\Omega_0$. 
Finally, to prove $(iii)$, we show first that $\xi$ is a homeomorphism between $\overline\Xi$ and $\overline\Gamma$, where $\overline \Xi$ is the closure of $\Xi$ in $\Sigma$. We have already seen that $\xi(\Xi)=\Gamma$ and, reasoning as in the verification of $(ii)$, we deduce that 
$\xi(\overline \Xi)\subseteq \overline \Gamma$. Moreover, $\xi \colon\overline \Xi\to \overline \Gamma$ is injective (since $\overline{\Xi} \subseteq \Sigma$ and thus $(\lambda,x_{1},x_{2})\in \overline{\Xi}$ implies that $x_{2}=\phi(x_{1}')$) and onto (since $\overline{\Gamma} \subseteq \xi(\Sigma)$ and so $\overline \Gamma = \xi(\overline \Xi)$). Then, $\xi$ is a bijection from $\overline \Xi$ onto $\overline \Gamma$. In addition, by the continuity of $\Psi^{-1}$, the map $\xi^{-1} \colon \overline \Gamma \to \overline \Xi$ is continuous and this proves that $\overline \Gamma$ is non-compact, otherwise $\overline \Xi$ should be. 
Then, the proof is complete.
\end{proof}

\begin{remark}\label{rem-biforc-both-direction}
The assumptions of Theorem~\ref{genbifthm} are not sufficient to determine whether the connected component $\Gamma$ defined in \eqref{definition of Gamma} contains only pairs $(0,x)$ with $x$ nonconstant, only pairs $(\lambda,x)$ with $\lambda\neq0$, or pairs of both types. 
In the particular case when the constant solutions of
\begin{equation*}
\begin{cases}
\, (\phi(x'))'=f_0(x,x'),
\\
\, x(0)=x(T),\quad x'(0)=x'(T),
\end{cases}
\end{equation*}
in $\Omega_0$ are isolated in $\Omega_0$, the bifurcating branch $\Gamma$ contains $(\lambda,x)$ with $\lambda\neq0$. 
\hfill$\lhd$
\end{remark}

\begin{remark}\label{rem-lambda-interval}
We emphasise that Theorem~\ref{genbifthm} have been presented under the assumption that problem \eqref{mainpb} is well defined for every $\lambda\in\mathbb{R}$, cf.~\ref{hp-H2}. However, we stress that it remains valid when the nonlinear term is the form
\begin{equation}\label{F-bounded-interval}
F\colon I\times\mathopen{[}0,T\mathclose{]}\times\mathbb{R}^n\times\mathbb{R}^n\to\mathbb{R}^n,
\end{equation}
where $I\subseteq \mathbb{R}$ is an open interval containing $0$. Although, for simplicity, the presentation has been carried out under the assumption $I=\mathbb{R}$, it extends naturally to this more general setting with only minor modifications.
\hfill$\lhd$
\end{remark}

Some corollaries are in order.

\begin{corollary}\label{corollario1} 
Let the assumptions of Theorem~\ref{genbifthm} be satisfied.
Then, \eqref{mainpb} admits a connected set of nontrivial solution pairs whose closure in $\mathbb{R}\times \mathcal{C}_T^1$ contains a trivial solution pair and either is unbounded or intersects $\partial\Omega$.
\end{corollary}

\begin{proof}
By Theorem~\ref{genbifthm}, let $\Gamma\subseteq \Omega$ be a connected set of nontrivial solution pairs whose closure $\overline \Gamma$ in $\Omega$ is not compact and intersects $\{0\} \times \widetilde\Omega_0$. 
If $\overline \Gamma\cap \partial\Omega=\emptyset$, then, the closure of $\Gamma$ in $\mathbb{R}\times \mathcal{C}_T^1$ coincides with $\overline \Gamma$. 
We claim that, in this case, $\overline{\Gamma}$ is unbounded.
Assume the contrary and suppose that $\overline{\Gamma}$ is bounded. 
Take, as in the proof of Theorem~\ref{genbifthm}, the map $ \xi\colon \Sigma\to \mathbb R \times \mathcal{C}^1_T$, given by \eqref{def-map-xi} and call $\overline \Xi \coloneqq \xi^{-1}(\overline\Gamma)$. 
As we have seen, $\xi$ is a homeomorphism between $\overline \Xi$ and $\overline\Gamma$. 
Since $\phi$ is a homeomorphism of $\mathbb{R}^n$, it sends bounded subsets of $\mathbb{R}^n$ into bounded subsets of $\mathbb{R}^n$, and this is clearly true also for its inverse $\phi^{-1}$. Consequently, $\overline \Xi$ turns out to be bounded since so is $\overline \Gamma$. However, as seen in the proof of Theorem~\ref{genbifthm} (see also Appendix~\ref{appendix-A}), the map 
\begin{equation*}
(\lambda,x_1,x_2) \mapsto (x_1,x_2) - \Phi(\lambda,x_1,x_2),
\end{equation*}
where $\Phi$ is defined as in \eqref{Phi  completely continuous}, is proper on closed and bounded subsets of $\mathbb{R} \times \mathcal{C}(\mathopen{[}0,T\mathclose{]},\mathbb{R}^n)\times \mathcal{C}(\mathopen{[}0,T\mathclose{]},\mathbb{R}^n)$, and thus $\overline \Xi$ is compact being a closed (in $\mathbb{R} \times \mathcal{C}(\mathopen{[}0,T\mathclose{]},\mathbb{R}^n)\times \mathcal{C}(\mathopen{[}0,T\mathclose{]},\mathbb{R}^n)$ and bounded subset of the zeros of the above map. This implies that $\overline \Gamma$ is compact as well, being $\xi$ a homeomorphism between $\overline \Xi$ and $\overline \Gamma$. This is a contradiction and we conclude that $\overline \Gamma$ is unbounded. Then, the proof is complete.
\end{proof}

\begin{remark}\label{rem-lambda-interval-bis}
In the case where $F$ is as in \eqref{F-bounded-interval}, with $I\subseteq \mathbb{R}$ an open and bounded interval, the statement of Corollary~\ref{corollario1} remains valid, provided that the boundary of $\Omega$ is considered with respect to the topology of $\mathbb{R}\times \mathcal{C}_T^1$, rather than that of $I\times \mathcal{C}_T^1$. This clarification is necessary, as the boundary $\partial\Omega$ may be empty in the latter topology (for instance, when $\Omega=I\times \mathcal{C}_T^1$). 
An illustrative case is provided in Example~\ref{example-5.3} which deals with a problem in which $F$ is defined in $(-1,1)\times \mathbb{R}^2$. In that case, there exists a bounded bifurcating branch of (nontrivial) solutions, whose closure in $\mathbb{R}\times \mathcal{C}_T^1$ contains points of  $\{-1,1\}\times \mathcal{C}_T^1$, that are not solutions pairs (since the problem is not defined for $\lambda\in \{-1,1\}$). 
\hfill$\lhd$
\end{remark}

We conclude this section with two other corollaries. The following result gives a sufficient condition to obtain a bifurcating branch of solutions of \eqref{mainpb} which is unbounded with respect to $\lambda$.
 
\begin{corollary}\label{corollario3} 
Let the assumptions of Theorem~\ref{genbifthm} be satisfied.
Suppose that the set of solutions of $(\phi(x'))'=f_0(x,x')$ is contained in $\Omega_0\coloneqq \{x\in  \mathcal{C}^1_T \colon (0,x)\in \Omega\}$.
Assume that there exists a continuous function $\mathcal{M}\colon\mathbb{R}\to [0,+\infty)$ such that $\Vert x\Vert_{\mathcal{C}^{1}}\leq \mathcal{M}(\lambda)$ for any solution pair $(\lambda,x)$ of \eqref{mainpb}. Then, \eqref{mainpb} admits a connected set of nontrivial solution pairs whose closure in $\mathbb{R}\times \mathcal{C}_T^1$ contains a trivial solution pair and is unbounded with respect to $\lambda$.
\end{corollary}

\begin{proof}
Let $\Omega'=(\{0\}\times \Omega_0 ) \cup ((\mathbb{R}\setminus\{0\}) \times \mathcal{C}^1_T)$. By Corollary~\ref{corollario1}, since $\partial \Omega'=\{0\}\times(\mathcal{C}^{1}_{T}\setminus\Omega_0)$, we obtain a connected set $\Gamma$ of nontrivial solution pairs of \eqref{mainpb}, whose closure $\overline \Gamma$ in $\mathbb{R}\times \mathcal{C}_T^1$ (this closure coincides with the closure in $\Omega'$) contains a trivial solution pair and is unbounded. 
Let $\lambda_0\in\mathbb{R}$ be given and assume that \eqref{mainpb} admits no solution for a particular $\lambda'>\lambda_{0}$. Therefore, by the connectedness of $\overline \Gamma$, the set $\overline \Gamma \cap([\lambda_0,+\infty)\times \mathcal{C}_T^1)$ is contained in $[\lambda_0,\lambda']\times \mathcal{C}_T^1$ and thus, by the properties of $\mathcal{M}$, it must be bounded. If \eqref{mainpb} also admits no solution for a particular $\lambda''<\lambda_{0}$, analogously, one can prove that $\overline \Gamma \cap((-\infty,\lambda_{0}]\times \mathcal{C}_T^1)$ is bounded.
This is a contradiction.
\end{proof}

\begin{remark}\label{rem-4.2}
Concerning the hypothesis about the limiting function $\mathcal{M}$, we stress that Corollary~\ref{corollario3} is also valid under the assumption that for every $\lambda>0$ there exists a continuous function $\mathcal{M}_{\lambda}\colon\mathopen{[}-\lambda,\lambda\mathclose{]}\to [0,+\infty)$ such that $\| x\|_{\mathcal{C}^{1}}\leq \mathcal{M}_{\lambda}(\lambda')$ for every solution pair $(\lambda',x)\in \mathopen{[}-\lambda,\lambda\mathclose{]}\times\mathcal{C}^{1}_{T}$.
\hfill$\lhd$
\end{remark}

The last corollary concerns bifurcating branches of solutions of \eqref{mainpb} which are unbounded with respect to the norm of $x$.

\begin{corollary}\label{corollario4} 
Let the assumptions of Theorem~\ref{genbifthm} be satisfied.
Let $\lambda_0>0$ be such that no solution pair exists with $\vert \lambda \vert >\lambda_0$. Then, for every $R>0$, problem \eqref{mainpb} admits a solution pair $(\lambda,x)$ with $\Vert x\Vert_{\mathcal{C}^{1}}>R$ and $\vert \lambda\vert \leq\lambda_0$.
\end{corollary}

\begin{proof}
Assume by contradiction that the set of solution pairs is bounded, that is, there exists $R^{*}>0$ such that $\| x \|_{\mathcal{C}^{1}} \leq R^*$ for every solution pair $(\lambda,x)$. Then, by choosing $\mathcal{M}(\lambda)=R^*$, we can apply Corollary~\ref{corollario3} obtaining a connected set of nontrivial solution pairs unbounded in $\lambda$, a contradiction. 
\end{proof}

\section{Some applications}\label{section-5}

We conclude the paper by presenting some applications of the results proved in Section~\ref{section-4}. In Subsection~\ref{section-5.1}, we illustrate an example where the branch of nontrivial solutions is unbounded with respect to the $\lambda$-component, by an application of Corollary~\ref{corollario3}. In Subsection~\ref{section-5.3}, we study a situation where the branch of nontrivial solutions is unbounded, but bounded in the $\lambda$-component. In Subsection~\ref{section-5.2}, we analyze a case of a bifurcating branch which is bounded. See Figure~\ref{fig-01} for a graphical representation of the results.

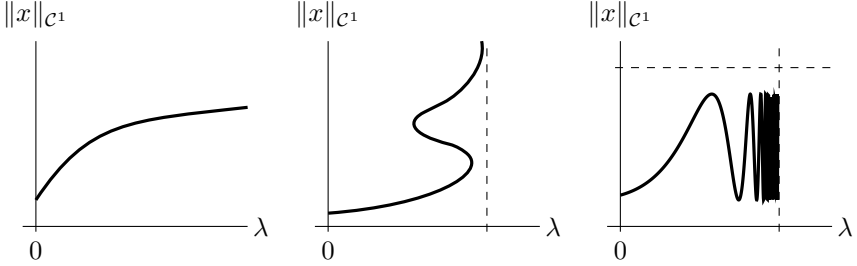
\begin{figure}[ht]
\centering
\begin{tikzpicture}[x=0.35cm,y=0.35cm]
\draw (-0.5,0) -- (8,0);
\draw (0,-0.2) -- (0,7);
\draw (0,8) node {$\|x\|_{\mathcal{C}^{1}}$};
\draw (0,-0.9) node {$0$};
\draw (8.5,0) node {$\lambda$};
\draw [line width=1.2pt, color = black, rounded corners]
(0,1) .. controls (2,4) and (3.5,4) .. (8,4.5);
\end{tikzpicture}
\begin{tikzpicture}[x=0.35cm,y=0.35cm]
\draw (-0.5,0) -- (8,0);
\draw[dashed] (6,-0.2) -- (6,7);
\draw (0,-0.2) -- (0,7);
\draw (0,8) node {$\|x\|_{\mathcal{C}^{1}}$};
\draw (0,-0.9) node {$0$};
\draw (8.5,0) node {$\lambda$};
\draw [line width=1.2pt, color = black, rounded corners]
(0,0.5) .. controls (4,0.8) and (6,2) .. 
(5,3) .. controls (3,3.5) and (3,4) ..
(4,4.5) .. controls (5,5) and (6,6) ..
(5.8,7);
\end{tikzpicture}
\begin{tikzpicture}[x=0.35cm,y=0.35cm]
\draw (-0.5,0) -- (8,0);
\draw[dashed] (6,-0.2) -- (6,7);
\draw[dashed] (-0.2,6) -- (8,6);
\draw (0,-0.2) -- (0,7);
\draw (0,8) node {$\|x\|_{\mathcal{C}^{1}}$};
\draw (0,-0.9) node {$0$};
\draw (8.5,0) node {$\lambda$};
\draw[line width=1.2pt, color = black, domain=0:5.95,smooth,samples=500] plot (\x, {2*sin((12/(\x-6))r)+3});
\end{tikzpicture}
\caption{Qualitative representation of the branches of nontrivial solutions in the three examples considered in Section~\ref{section-5}.}
\label{fig-01}
\end{figure}

We remark that, in order to apply Theorem~\ref{genbifthm} and its corollaries, it is necessary to verify the topological degree condition \eqref{gradononzeroo}. Although computing the Brouwer degree may become intricate in settings of dimension greater than one, the examples considered in this section involve particular classes of mappings for which the degree can be readily determined. For instance, in the case of an odd mapping defined on a symmetric neighbourhood of the origin the classical Borsuk theorem asserts that it has an odd Brouwer degree. We will exploit more general formulations of this result, cf.~\cite[Section~7.3.1]{DiMa-21}.

\subsection{A branch of periodic solutions unbounded in the $\lambda$-component}\label{section-5.1}

We consider the following $T$-periodic boundary value problem associated with a $\phi$-Laplacian Li\'enard-type equation
\begin{equation}\label{eq-example-1}
\begin{cases}
\, ( \phi(x') )' = \ell(\lambda) \dfrac{\mathrm{d}}{\mathrm{d}t} \nabla \mathcal{G}(x) + h(\lambda,t,x),
\vspace*{2pt}
\\
\, x(0) = x(T), \quad x'(0) = x'(T).
\end{cases}
\end{equation}
We assume that $\phi \colon \mathbb{R}^{n} \to \phi(\mathbb{R}^{n})=\mathbb{R}^{n}$ is a homeomorphism such that
\begin{enumerate}[leftmargin=26pt,labelsep=6pt,label=\textup{$(\textsc{a}_{1})$}]
\item $\phi$ is of the form
\begin{equation*}
\phi(\xi) = A(\xi)\xi,
\end{equation*}
where $A\colon \mathbb{R}^{n} \to \mathopen{[}0,+\infty)$ is a continuous function with $A(0)=0$ and $A(\xi)>0$ for $\xi\neq0$.
\label{phi=A}
\end{enumerate}
This assumption includes a wide quantity of nonlinear differential operators considered in the literature, in particular the vector $p$-Laplacians $\phi(\xi)=\Vert\xi\Vert^{p-1}\xi$.
Notice also that
\begin{equation}\label{eq-phi-positive}
\langle \phi(\xi),\xi \rangle = A(\xi) \|\xi\|^{2} = \|\phi(\xi)\| \|\xi\| \geq 0,\quad \text{for every $\xi\in\mathbb{R}^{n}$,}
\end{equation}
and thus, since $\|\phi(\xi)\|\to+\infty$ as $\Vert\xi\Vert\to+\infty$, there exists $\gamma >0$ such that
\begin{equation}\label{prop-A0}
\langle \phi(\xi),\xi \rangle \geq \|\xi\| - \gamma, \quad \text{for every $\xi\in\mathbb{R}^{n}$}
\end{equation}
(compare also with condition $(H_{2})$ in \cite{MaMa-98}).

Let $\ell\colon\mathbb{R}\to\mathbb{R}$ be continuous with $\ell(0)=0$, and $\mathcal{G}\colon\mathbb{R}^{n}\to\mathbb{R}$ be a function of class $\mathcal{C}^{2}$. Let $h=h(\lambda,t,x) \colon \mathbb{R}\times\mathopen{[}0,T\mathclose{]}\times\mathbb{R}^{n}\to \mathbb{R}^{n}$ be a Carath\'eodory function (cf.~hypothesis \ref{hp-H2}) such that
\begin{enumerate}[leftmargin=26pt,labelsep=6pt,label=\textup{$(\textsc{a}_{2})$}]
\item $h(0,t,x)=h_{0}(x)$ for a.e.~$t\in\mathopen{[}0,T\mathclose{]}$ and every $x\in\mathbb{R}^{n}$;
\label{hp-A2}
\end{enumerate}
\begin{enumerate}[leftmargin=26pt,labelsep=6pt,label=\textup{$(\textsc{a}_{3})$}]
\item there exists $R > 0$ such that $\langle h(\lambda,t,x),x\rangle > 0$, for every $\lambda\in\mathbb{R}$, a.e.~$t\in\mathopen{[}0,T\mathclose{]}$, and every $x\in\mathbb{R}^{n}$ with $\|x\| \geq R$.
\label{hp-A3}
\end{enumerate}
We point out that in this setting hypotheses \ref{hp-H1}, \ref{hp-H2}, \ref{hp-H3} are satisfied.

\begin{theorem}\label{th-ex-1}
Let $\phi$, $\ell$, $\mathcal{G}$, and $h$ be as above. Then, problem \eqref{eq-example-1} admits a bifurcating branch of solution pairs unbounded with respect to $\lambda$.
\end{theorem}

\begin{proof}
We aim to show that Corollary~\ref{corollario3} applies.
We divide the proof in some steps.

\smallskip
\noindent
\textit{Step~1. A priori bound of $\|x'\|_{L^{1}}$.} Let $(\lambda,x)$ be a solution pair of \eqref{eq-example-1}. We integrate the scalar product between $x$ and the equation in \eqref{eq-example-1} to obtain 
\begin{equation}\label{eq-3.6}
\int_{0}^{T} \langle \phi(x'(t)),x'(t) \rangle \,\mathrm{d}t 
+ \ell(\lambda) \int_{0}^{T} \langle \dfrac{\mathrm{d}}{\mathrm{d}t} \nabla \mathcal{G}(x(t)),x(t)\rangle\,\mathrm{d}t 
+ \int_{0}^{T}\langle h(\lambda,t,x(t)),x(t)\rangle\,\mathrm{d}t 
=0.
\end{equation}
By the $T$-periodicity of $x$, we have
\begin{equation*}
\int_{0}^{T} \langle \dfrac{\mathrm{d}}{\mathrm{d}t} \nabla \mathcal{G}(x(t)),x(t)\rangle\,\mathrm{d}t 
=- \int_{0}^{T}\langle \nabla \mathcal{G}(x(t)),x'(t)\rangle\,\mathrm{d}t = -\mathcal{G}(x(T))+\mathcal{G}(x(0))=0.
\end{equation*}
Next, we observe that condition \ref{hp-A3} implies the existence of a continuous and positive function $K_{0}$ such that $\langle h(\lambda,t,s),s\rangle \geq -K_{0}(\lambda)$, for every $\lambda\in \mathbb{R}$, a.e.~$t\in\mathopen{[}0,T\mathclose{]}$, and every $s\in\mathbb{R}^{n}$. Therefore, from \eqref{eq-3.6}, we deduce that
\begin{equation*}
\int_{0}^{T} \langle \phi(x'(t)),x'(t) \rangle \,\mathrm{d}t 
= -\int_{0}^{T}\langle h(\lambda,t,x(t)),x(t)\rangle\,\mathrm{d}t
\leq TK_{0} (\lambda).
\end{equation*}
By \eqref{prop-A0}, we obtain that
\begin{equation*}
\|x'\|_{L^{1}} = \int_{0}^{T} \|x'(t)\|\,\mathrm{d}t \leq TK_{0}(\lambda)+T\gamma.
\end{equation*}

\smallskip
\noindent
\textit{Step~2. A priori bound of $\|x\|_{\infty}$.}
Let $(\lambda,x)$ be a solution pair of \eqref{eq-example-1}.
First, we prove that there exists $\tilde{t}\in\mathopen{[}0,T\mathclose{]}$ such that $\|x(\tilde{t})\|<R$, where $R>0$ is defined as in \ref{hp-A3}. Indeed, if it is false, $\|x(t)\|\geq R$ for all $t\in\mathopen{[}0,T\mathclose{]}$. Next, from \eqref{eq-3.6}, \eqref{eq-phi-positive} and \ref{hp-A3}, we find that
\begin{equation*}
0 = \int_{0}^{T} \langle \phi(x'(t)),x'(t) \rangle \,\mathrm{d}t 
+ \int_{0}^{T}\langle h(\lambda,t,x(t)),x(t)\rangle\,\mathrm{d}t > 0,
\end{equation*}
a contradiction. Consequently, for every $t\in\mathopen{[}0,T\mathclose{]}$, we immediately obtain
\begin{equation*}
\|x(t)\|
= \biggl{\|} x(\tilde{t}) + \int_{\tilde{t}}^{t} x'(s)\,\mathrm{d}s \biggr{\|}
\leq \|x(\tilde{t})\| + \int_{0}^{T}\|x'(s)\|\,\mathrm{d}s 
< R+TK_{0}(\lambda)+T\gamma,
\end{equation*}
thus an a priori bound for $\|x\|_{\infty}$. 

\smallskip
\noindent
\textit{Step~3. The limiting functions.}
Let $\lambda>0$ be given. Denote
\begin{equation*}
K(\lambda) \coloneqq \max_{\lambda'\in\mathopen{[}-\lambda,\lambda\mathclose{]}} K_{0}(\lambda'),
\qquad
\hat{H}(\lambda) \coloneqq \max_{\|\xi\|\leq R+TK(\lambda)+T\gamma} \|\mathrm{Hess}\,\mathcal{G}(\xi)\|_{\infty}.
\end{equation*}
Since $h$ is Carath\'{e}odory, there exists a map $g_{\lambda}\in L^1(\mathopen{[}0,T\mathclose{]},[0,+\infty))$ such that, for a.e.~$t\in \mathopen{[}0,T\mathclose{]}$ and every $(\lambda',\xi) \in \mathopen{[}-\lambda,\lambda\mathclose{]}\times \mathbb{R}^n$, with $\Vert \xi\Vert\leq R+TK(\lambda)+T\gamma$, we have $\| h(\lambda',t,\xi)\|\leq g_{\lambda}(t)$.
Let $(\lambda',x)$ be a solution pair of \eqref{eq-example-1} with $\lambda'\in\mathopen{[}-\lambda,\lambda\mathclose{]}$. From
\begin{align*}
\|(\phi(x'))'\|_{L^{1}}
&= \int_{0}^{T} \|(\phi(x'(t)))'\|\,\mathrm{d}t
\leq |\ell(\lambda')| \, \hat{H}(\lambda) \, \|x'\|_{L^{1}}+ \int_{0}^{T} g_{\lambda}(t)\,\mathrm{d}t 
\\
&\leq |\ell(\lambda')| \, \hat{H}(\lambda) \, (TK(\lambda)+T\gamma)+ \|g_{\lambda}\|_{L^{1}} \eqqcolon \mathcal{P}_{\lambda}(\lambda').
\end{align*}
Notice that $\mathcal{P}_{\lambda}$, which is defined in $\mathopen{[}-\lambda,\lambda\mathclose{]}$, is continuous since so is $\ell$. 

We claim that
\begin{align*}
\|\phi(x'(\cdot))\|_{\infty} \leq \mathcal{P}_{\lambda}(\lambda').
\end{align*}
Assume that $x$ is not constant (i.e., $x'\not\equiv0$), otherwise the estimate is trivially satisfied.
Let $\hat{t}\in\mathopen{[}0,T\mathclose{]}$ be such that $\|\phi(x'(\hat{t}))\|=\|\phi(x'(\cdot))\|_{\infty}\neq0$. Via an application of Rolle's theorem to $t\mapsto \langle x(t),x'(\hat{t})\rangle$, we have that there exists $t_{0}\in\mathopen{[}0,T\mathclose{]}$ such that $\langle x'(t_{0}),x'(\hat{t}) \rangle=0$ and thus $\langle \phi(x'(t_{0})),\phi(x'(\hat{t}))\rangle=0$, by \ref{phi=A}. Therefore, we have
\begin{align*}
\|\phi(x'(\cdot))\|_{\infty}
&=\|\phi(x'(\hat{t}))\|
=\bigl{|} \langle \phi(x'(\hat{t})),\phi(x'(\hat{t}))\rangle \bigr{|} /\|\phi(x'(\cdot))\|_{\infty}  
\\
&= \biggl{|} \langle \phi(x'(t_{0})),\phi(x'(\hat{t}))\rangle +\int_{t_{0}}^{\hat{t}} \langle (\phi(x'(s)))',\phi(x'(\hat{t}))\rangle \,\mathrm{d}s \biggr{|}/\|\phi(x'(\cdot))\|_{\infty} 
\\
&\leq \mathcal{P}_{\lambda}(\lambda'),
\end{align*}
where in the last inequality we exploit the Cauchy--Schwarz one.
This ensures an a priori bound for $\|x'\|_{\infty}$. As a consequence of the a priori bound for $\|x\|_{\infty}$ in Step~2, we obtain the existence of a continuous function $\mathcal{M}_{\lambda}$ as in Remark~\ref{rem-4.2}.

\smallskip
\noindent
\textit{Step~4. Conclusion.}
Let $\Omega = (\{0\}\times B(0,r)) \cup((\mathbb{R}\setminus\{0\})\times\mathcal{C}^{1}_{T})$, where \begin{equation*}
r> R+TK_{0}(0)+T\gamma
\end{equation*}
and $B(0,r)$ denotes the open ball in $\mathcal{C}^{1}_{T}$ centered at zero and with radius $r$.
By \ref{hp-A3}, we deduce that $h_{0}(-\xi) \neq \mu \, h_{0}(\xi)$, for every $\xi\in\mathbb{R}^{n}$ with $\|\xi\|=r$ and every $\mu\geq1$. As a consequence, we can apply an odd mapping-type theorem (cf.~\cite[Corollary~7.3.1]{DiMa-21}) to $h_{0}$ in the open ball $\mathcal{B}(0,r)$ in $\mathbb{R}^{n}$ centered at zero and with radius $r$, obtaining that the degree $\mathrm{deg}_{\mathrm{B}}(h_{0},\mathcal{B}(0,r), 0)$ is well defined and non-zero. Therefore, the thesis follows from Corollary~\ref{corollario3}.
\end{proof}

\begin{example}\label{example-5.1}
In order to better illustrate the hypotheses assumed above, we present a very simple situation in which Theorem~\ref{th-ex-1} applies.

Let us consider the $T$-periodic problem associated with the planar system
\begin{equation*}
\begin{cases}
\, ( (|x_{1}'|^{2}+|x_{2}'|^{2}) x_{1}' )' = a_{1}(x_{1},x_{2}) x_{1} + |\lambda| e_{1}(t) x_{1},
\\
\, ( (|x_{1}'|^{2}+|x_{2}'|^{2}) x_{2}' )' = a_{2}(x_{1},x_{2}) x_{2} + |\lambda| e_{2}(t) x_{2},
\end{cases}
\end{equation*}
where $a_{1},a_{2}\in\mathcal{C}(\mathbb{R}\times\mathbb{R}, \mathopen{(}0,+\infty\mathclose{)})$, $e_{1},e_{2}\in\mathcal{C}(\mathopen{[}0,T\mathclose{]},\mathopen{[}0,+\infty\mathclose{)})$, $\lambda\in\mathbb{R}$. This problem corresponds to \eqref{eq-example-1} with $\phi(\xi)=\Vert\xi\Vert^{2}\xi$ (i.e., the vector $3$-Laplacian operator) and $\mathcal{G}\equiv0$.
It is easy to check that hypotheses \ref{phi=A} and \ref{hp-A2} hold true; while \ref{hp-A3} is a consequence of the sign conditions on $a_{1},a_{2},e_{1},e_{2}$.
\end{example}

\subsection{An unbounded branch of periodic solutions, bounded in the $\lambda$-component}\label{section-5.3}

We consider the following $T$-periodic boundary value problem
\begin{equation}\label{eq-example-2}
\begin{cases}
\, ( \phi(x') )' = \lambda \dfrac{\mathrm{d}}{\mathrm{d}t} \nabla \mathcal{G}(x) + h_{0}(x) + \lambda e(t),\\
\, x(0) = x(T), \quad x'(0) = x'(T),
\end{cases}
\end{equation}
where $\phi=(\phi_{1},\ldots,\phi_{n}) \colon \mathbb{R}^{n} \to \phi(\mathbb{R}^{n})=\mathbb{R}^{n}$ is a homeomorphism such that $\phi(0)=0$ and
\begin{enumerate}[leftmargin=26pt,labelsep=6pt,label=\textup{$(\textsc{a}_{4})$}]
\item $\langle \phi(\xi),\xi \rangle >0$, for every $\xi\in\mathbb{R}^{n}\setminus\{0\}$.
\label{hp-A4}
\end{enumerate}
Let $\mathcal{G}\colon\mathbb{R}^{n}\to\mathbb{R}$ be a function of class $\mathcal{C}^{2}$. Let $h_{0}= (\mathfrak{h}_{1},\ldots,\mathfrak{h}_{n}) \colon \mathbb{R}^{n} \to \mathbb{R}^{n}$ be a continuous function satisfying \begin{enumerate}[leftmargin=26pt,labelsep=6pt,label=\textup{$(\textsc{a}_{5})$}]
\item $\langle h_{0}(\xi),\xi \rangle >0$, for every $\xi\in\mathbb{R}^{n}\setminus\{0\}$.
\label{hp-A5}
\end{enumerate}
Assume that $e \colon \mathopen{[}0,T\mathclose{]} \to \mathbb{R}^{n}$ is integrable and such that
\begin{enumerate}[leftmargin=26pt,labelsep=6pt,label=\textup{$(\textsc{a}_{6})$}]
\item there exists $i \in\{1,\ldots,n\}$ such that $\sup_{\mathbb{R}^{n}} |\mathfrak{h}_{i}|<+\infty$ and $\int_{0}^{T} e_{i}(t) \,\mathrm{d}t \neq 0$.
\label{hp-A6}
\end{enumerate}
We stress that in this setting hypotheses \ref{hp-H1}, \ref{hp-H2}, \ref{hp-H3} are satisfied.

\begin{theorem}\label{th-ex-3}
Let $\phi$, $\mathcal{G}$, $h_{0}$, and $e$ be as above. Then, problem \eqref{eq-example-2} admits a bifurcating branch of solution pairs which emanates from the origin and is unbounded with respect to $x$.
\end{theorem}

\begin{proof}
First, by exploiting \ref{hp-A6}, we show that there is no solution pair $(\lambda,x)$ of problem \eqref{eq-example-2} if $|\lambda|$ is sufficiently large.

Without loss of generality, assume that $\int_{0}^{T} e_{i}(t) \,\mathrm{d}t>0$, being the other case analogous. By an integration of the $i$-th component of the equation in \eqref{eq-example-2}, we deduce that
\begin{equation*}
0 = \int_{0}^{T} (\phi_{i}(x'(t)))' \,\mathrm{d}t = \int_{0}^{T} \mathfrak{h}_{i}(x(t))\,\mathrm{d}t + \lambda \int_{0}^{T}e_{i}(t)\,\mathrm{d}t.
\end{equation*}
Then,
\begin{equation*}
\lambda \int_{0}^{T}e_{i}(t)\,\mathrm{d}t = -\int_{0}^{T} \mathfrak{h}_{i}(x(t))\,\mathrm{d}t \in \left[ -T \sup_{\mathbb{R}^{n}} |\mathfrak{h}_{i}|,T \sup_{\mathbb{R}^{n}} |\mathfrak{h}_{i}| \right].
\end{equation*}
We conclude that for $|\lambda|>\hat{\lambda}_{i}$, with
\begin{equation}\label{def-hat-lambda}
\hat{\lambda}_{i}\coloneqq\dfrac{T \,  \sup_{\mathbb{R}^{n}} |\mathfrak{h}_{i}|}{\int_{0}^{T}e_{i}(t)\,\mathrm{d}t}>0,
\end{equation}
no solution pair $(\lambda,x)$ of \eqref{eq-example-2} exists. 

Second, we focus on problem \eqref{eq-example-2} for $\lambda=0$, that is,
\begin{equation}\label{eq-example-3}
\begin{cases}
\, ( \phi(x') )' = h_{0}(x),
\\
\, x(0) = x(T), \quad x'(0) = x'(T).
\end{cases}
\end{equation}
We are going to prove that the identically zero function is the only solution of \eqref{eq-example-3}.
We observe that $h_{0}(\xi)$ vanishes only at $\xi=0$, by \ref{hp-A5}.
Therefore, if $x$ is a nontrivial solution of problem \eqref{eq-example-3}, then $x$ is not constant. Hence, recalling \ref{hp-A4} and \ref{hp-A5}, and integrating, we obtain
\begin{equation}
0 > - \int_{0}^{T} \langle \phi(x'(t) ), x'(t) \rangle \,\mathrm{d}t = \int_{0}^{T} \langle h_{0}(x(t)),x(t) \rangle \,\mathrm{d}t >0,
\end{equation}
a contradiction.

Let $\Omega\coloneqq(-\hat{\lambda}_{i}-1,\hat{\lambda}_{i}+1)\times \mathcal{C}^{1}_{T}$, where $\hat{\lambda}_{i}$ is defined as in \eqref{def-hat-lambda}. Analogously to Step~4 of the proof of Theorem \ref{th-ex-1}, we infer that the degree $\mathrm{deg}_{\mathrm{B}}(h_{0}, \mathbb{R}^{n}, 0) = \mathrm{deg}_{\mathrm{B}}(h_{0}, \mathcal{B}(0,r), 0) \neq 0$ for every $r>0$ by \ref{hp-A5}, thus the hypotheses of Theorem~\ref{genbifthm} are satisfied.
An application of Corollary~\ref{corollario1} provides a bifurcating branch in $\mathbb{R} \times \mathcal{C}^{1}_{T}$ of solution pairs of \eqref{eq-example-2} emanating from $(0,0)$, which either is unbounded or intersects $\partial\Omega$. Actually, by the choice of $\Omega$, the second situation does not occur. Hence, we conclude that such a branch bifurcates from the trivial solution pair $(0,0)$ and is unbounded with respect to the $x$-component.
\end{proof}

\begin{example}\label{example-5.2}
Analogously to the previous subsection, we show here a simple example in which Theorem~\ref{th-ex-3} applies.
Let us consider the $T$-periodic problem associated with the planar system
\begin{equation*}
\begin{cases}
\, ( (|x_{1}'|^{2}+|x_{2}'|^{2}) x_{1}' )' = a_{1}(x_{1},x_{2}) x_{1} + \lambda e_{1}(t),
\\
\, ( (|x_{1}'|^{2}+|x_{2}'|^{2}) x_{2}' )' = a_{2}(x_{1},x_{2}) x_{2} + \lambda e_{2}(t),
\end{cases}
\end{equation*}
where $a_{1},a_{2}\in\mathcal{C}(\mathbb{R}\times\mathbb{R}, \mathopen{(}0,+\infty\mathclose{)})$, $e_{1},e_{2}\in\mathcal{C}(\mathopen{[}0,T\mathclose{]},\mathbb{R})$, $\lambda\in\mathbb{R}$. This problem corresponds to \eqref{eq-example-2} with $\phi(\xi)=\Vert\xi\Vert^{2}\xi$ and $\mathcal{G}\equiv0$.
It is easy to check that hypotheses \ref{hp-A4} and \ref{hp-A5} hold true. In order to verify hypothesis \ref{hp-A6}, for instance we can take
\begin{equation*}
a_{1}(x_{1},x_{2}) = \dfrac{1}{|x_{1}|+1}, \qquad 
\int_{0}^{T} e_{1}(t) \,\mathrm{d}t \neq 0,
\end{equation*}
hence $\sup_{\mathbb{R}^{2}} |\mathfrak{h}_{1}| = 1$.
\end{example}

\subsection{A bounded branch of periodic solutions}\label{section-5.2}

We consider the following $T$-periodic boundary value problem
\begin{equation}\label{eq-example-4}
\begin{cases}
\, ( \phi(x') )' = h_{0}(x) + g(\lambda) e(t),\\
\, x(0) = x(T), \quad x'(0) = x'(T).
\end{cases}
\end{equation}
As in Subsection~\ref{section-5.1}, we assume that $\phi \colon \mathbb{R}^{n} \to \phi(\mathbb{R}^{n})=\mathbb{R}^{n}$ is a homeomorphism satisfying \ref{phi=A}.
Moreover, let $h_{0}= (\mathfrak{h}_{1},\ldots,\mathfrak{h}_{n}) \colon \mathbb{R}^{n} \to \mathbb{R}^{n}$ be a continuous function satisfying
\begin{enumerate}[leftmargin=26pt,labelsep=6pt,label=\textup{$(\textsc{a}_{7})$}]
\item there exist $R_{0}>0$, $C_{0}>0$, and $\sigma>1$ such that $\langle h_{0}(\xi),\xi \rangle \geq C_{0} \|\xi\|^{\sigma}$ for every $\xi\in\mathbb{R}^{n}$ with $\|\xi\|\geq R_{0}$;
\label{hp-A7}
\end{enumerate}
\begin{enumerate}[leftmargin=26pt,labelsep=6pt,label=\textup{$(\textsc{a}_{8})$}]
\item there exists $\delta>0$ such that 
\begin{equation*}
h_{0}(-\xi)\neq \mu \, h_{0}(\xi), 
\quad 
\text{for every $\mu \geq 1$ and for every $\xi\in\mathbb{R}^{n}$ with $\|\xi\|=\delta$,}
\end{equation*}
and there exists $p\in\mathbb{R}^{n}$ with $\|p\|<\delta$ such that
\begin{equation*}
\bigl{\{} \xi\in \mathbb{R}^{n} \colon h_{0}(\xi) = 0, \; \|\xi\|\leq\delta \bigr{\}} = \{p\}.
\end{equation*}
\label{hp-A8}
\end{enumerate}
Let $g \colon I \to J$ be a continuous function, where $I,J \subseteq \mathbb{R}$ are open bounded intervals with $0\in I$.
Let $e \colon \mathopen{[}0,T\mathclose{]} \to \mathbb{R}^{n}$ be an essentially bounded measurable function.

We observe that the assumption on the function $g$ implies that problem \eqref{eq-example-4} is well-defined only for parameters $\lambda\in I$, where $I \subsetneq \mathbb{R}$ is a proper subset. 
This fact ensures that if $(\lambda,x)$ is a solution pair of problem \eqref{eq-example-4}, then $\lambda$ is bounded; cf.~Remark~\ref{rem-bound-lambda}. 

Moreover, recalling the discussion in Remark~\ref{rem-lambda-interval}, hypotheses \ref{hp-H1}, \ref{hp-H2}, \ref{hp-H3} hold true.

\begin{theorem}\label{th-ex-2}
Let $\phi$, $h_{0}$, $g$ and $e$ be as above. Then, problem \eqref{eq-example-4} admits a bifurcating branch of solution pairs $\Gamma \subseteq I \times \mathcal{C}^{1}_{T}$ which emanates from the point $(0,p)$, is bounded both with respect to $\lambda$ and $x$, and its closure $\overline{\Gamma}$ in $\mathbb{R}\times\mathcal{C}^{1}_{T}$ satisfies at least one of the following scenarios:
\begin{itemize}
\item $\overline{\Gamma}$ contains a point $(0,\bar{x})$ with $\bar{x}\neq p$;
\item $\overline{\Gamma}$ contains a point in $\{\inf I, \sup I \} \times \mathcal{C}^{1}_{T}$.
\end{itemize}
\end{theorem}

\begin{proof}
Our first goal is to find an a priori bound for $\|x\|_{\mathcal{C}^{1}}$, for any solution pair $(\lambda,x)$ of \eqref{eq-example-4}. We divide the proof in some steps.

\smallskip
\noindent
\textit{Step~1. A priori bound of $\|x\|_{\infty}$.} 
According to hypothesis \ref{hp-A7}, let $\hat{R}\geq \max\{R_{0},\delta\}$ be such that
\begin{equation*}
C_{0} \hat{R}^{\sigma-1} - \hat{g} \, \|e\|_{L^{\infty}} > 0,
\qquad
\text{where } \hat{g} \coloneqq \sup_{\lambda\in I} |g(\lambda)|.
\end{equation*}
For a.e.~$t\in\mathopen{[}0,T\mathclose{]}$, every $\lambda\in I$, and every $\xi\in\mathbb{R}^{n}$ with $\|\xi\|\geq \hat{R}$, we have that
\begin{equation}\label{eq-estimate}
\langle h_{0}(\xi)+g(\lambda) e(t), \xi \rangle 
\geq C_{0} \|\xi\|^{\sigma} - |g(\lambda)| |\langle e(t), \xi \rangle| 
\geq \|\xi\| (C_{0} \hat{R}^{\sigma-1} - \hat{g} \, \|e\|_{L^{\infty}}) 
> 0.
\end{equation}
Let $R>\hat{R}$. We claim that $\|x\|_{\infty}\neq R$, for every solution pair $(\lambda,x)$ of \eqref{eq-example-4}.
By contradiction, assume that there exists a solution pair $(\lambda,x)$ of \eqref{eq-example-4} such that $\|x\|_{\infty}= R$. Let $t^{*}\in\mathopen{[}0,T\mathclose{]}$ be such that $\|x(t^*)\|=\|x\|_{\infty}=R$. Since $t^{*}$ is a maximum point of $t\mapsto \|x(t)\|^{2}=\langle x(t),x(t) \rangle$, then it is a critical point of that map, that is $\langle x'(t^{*}),x(t^{*}) \rangle = 0$. By \ref{phi=A}, we deduce that $\langle \phi(x'(t^{*})),x(t^{*}) \rangle = 0$. Moreover,
\begin{equation*}
\dfrac{\mathrm{d}}{\mathrm{d}t} \langle \phi(x'(t)),x(t) \rangle = \langle h_{0}(x(t))+g(\lambda) e(t) ,x(t) \rangle + \langle \phi(x'(t)),x'(t) \rangle,
\quad \text{for a.e.~$t\in\mathopen{[}0,T\mathclose{]}$.}
\end{equation*}
Let $\varepsilon>0$ be such that $\|x(t)\| \geq \hat{R}$ for every $t\in\mathopen{[}t^{*}-\varepsilon,t^{*}+\varepsilon\mathclose{]}$. As a consequence of \eqref{eq-estimate}, we have
\begin{equation*}
\langle h_{0}(x(t))+g(\lambda) e(t) ,x(t) \rangle > 0,
\quad \text{for a.e.~$t\in\mathopen{[}t^{*}-\varepsilon,t^{*}+\varepsilon\mathclose{]}$.}
\end{equation*}
We deduce that the absolutely continuous function $v(t)\coloneqq \langle \phi(x'(t)),x(t) \rangle$ satisfies $v(t^{*})=0$ and 
\begin{equation*}
v'(t)>0, \quad \text{for a.e.~$t\in\mathopen{[}t^{*}-\varepsilon,t^{*}+\varepsilon\mathclose{]}$.}
\end{equation*}
Therefore, $v$ is strictly increasing in $\mathopen{[}t^{*}-\varepsilon,t^{*}+\varepsilon\mathclose{]}$, thus $v(t) < 0$ for all $t\in[t^{*}-\varepsilon,t^{*})$, and $v(t)> 0$ for all $t\in(t^{*},t^{*}+\varepsilon]$. We conclude that $t^{*}$ is not a maximum point of $t\mapsto \|x(t)\|^{2}$. We have reached a contradiction and the claim is proved.

\smallskip
\noindent
\textit{Step~2. A priori bound of $\|x'\|_{\infty}$.} Let $(\lambda,x)$ be a solution pair of \eqref{eq-example-4}. From the previous step we know that $\|x\|_{\infty}< R$.
From
\begin{equation*}
\dfrac{1}{T} \int_{0}^{T} \langle \phi(x'), x'(t) \rangle \,\mathrm{d}t = - \dfrac{1}{T} \int_{0}^{T} \langle h_{0}(x(t)) + g(\lambda) e(t) ,x(t) \rangle  \,\mathrm{d}t \leq KR,
\end{equation*}
where $K \coloneqq \max\{\|h_{0}(\xi)\| \colon \xi\in \mathcal{B}[0,R]\} + \hat{g}\,\|e\|_{L^{\infty}}$, by the mean value theorem we deduce that there exists $t_{0}\in\mathopen{[}0,T\mathclose{]}$ such that $A(x'(t_{0})) \|x'(t_{0})\|^{2} = \langle \phi(x'(t_{0})), x'(t_{0}) \rangle \leq KR$. Since $\|\phi(\xi)\|\to+\infty$ as $\Vert\xi\Vert\to+\infty$, let $L_{R}>0$ be such that if $\|\xi\|>L_{R}$ then $A(\xi) \|\xi\|^{2} > KR$. As a consequence, $\|x'(t_{0})\|\leq L_{R}$.
By an integration we have that
\begin{equation*}
\|\phi(x'(t))\| \leq \| \phi(x'(t_{0}))\| + \int_{0}^{T} \|h_{0}(x(t)) + g(\lambda) e(t)\| \,\mathrm{d}t \leq L^{\phi}_{R} + T K,
\quad \text{for every $t\in\mathopen{[}0,T\mathclose{]}$,}
\end{equation*}
where $L^{\phi}_{R}\coloneqq \max \{\| \phi(\xi)\| \colon \xi \in \mathcal{B}[0,L_{R}]\}$.
Hence $x'(t)\in\phi^{-1}(\mathcal{B}[0,L^{\phi}_{R} + T K])\subseteq \mathcal{B}[0,M_{R}]$ for every $t\in\mathopen{[}0,T\mathclose{]}$, for some $M_{R}>0$.
We conclude that $\|x'\|_{\infty} \leq M_{R}$.

\smallskip
\noindent
\textit{Step~3. Conclusion.} Let $(\lambda,x)$ be a solution pair of \eqref{eq-example-4}. Obviously $\lambda$ belongs to $I$, which is a bounded interval. From the previous discussion, we find that $\|x\|_{\mathcal{C}^{1}} = \|x\|_{\infty} + \|x'\|_{\infty} \leq R+M_{R}$.
Furthermore, exploiting \ref{hp-A8}, we apply an odd mapping-type theorem to $h_{0}$, obtaining that the degree $\mathrm{deg}_{\mathrm{B}}(h_{0},\mathcal{B}(0,\delta), 0)$ is well defined and non-zero (cf.~\cite[Corollary~7.3.1]{DiMa-21}).

Next, we define
\begin{equation*}
\Omega = (\{0\}\times B(0,\delta)) \cup ((I\setminus\{0\})\times B(0,R+M_{R}+1)),
\end{equation*}
which is an open set since $\delta < R+M_{R}+1$. Then, by \ref{hp-A8}, the hypotheses of Theorem~\ref{genbifthm} are satisfied. 
An application of Corollary~\ref{corollario1}, see also Remark~\ref{rem-lambda-interval-bis}, provides a bifurcating branch $\Gamma$ of solution pairs of \eqref{eq-example-4} whose closure $\overline{\Gamma}$, being bounded in both components, contains $(0,p)$ and intersects $\partial \Omega$.
We notice that, for every $(\lambda,x)\in\overline{\Gamma}$, we have $x\notin \partial B(0,\delta)$ if $\lambda=0$ and $x\notin \partial B(0,R+M_{R}+1)$ if $\lambda\neq 0$. Therefore, the alternative trivially holds and the proof is completed.
\end{proof}

\begin{remark}\label{rem-bound-lambda}
If one assumes an additional condition -- compatible with the other hypotheses -- that ensures the boundedness of $\lambda$ (in the spirit of \ref{hp-A6}), then the case $g(\lambda)=\lambda$, corresponding to $I=\mathbb{R}$ as in problem~\eqref{eq-example-2}, can also be considered. In this setting, the boundedness of $\lambda$ directly implies the boundedness of $g$. 

Incidentally, we emphasise that hypotheses \ref{hp-A6} and \ref{hp-A7} are mutually incompatible. Indeed, if a component $\mathfrak{h}_{i}$ of $h_{0}$ is bounded, then $\mathfrak{h}_{i}(\xi) \xi_{i} / \|\xi\|^{\sigma}$ tends to zero as $\|\xi\|\to+\infty$, provided $\sigma>1$. Consequently, assumption~\ref{hp-A7} cannot hold, for instance, when $\xi$ has all components equal to zero except for the $i$-th one.
\hfill$\lhd$
\end{remark}

\begin{example}\label{example-5.3}
In order to clarify the hypotheses assumed above, we discuss a simple situation in which Theorem~\ref{th-ex-2} applies, in the same line of Example~\ref{example-5.1}.

Let us consider the $T$-periodic problem associated with the planar system
\begin{equation}\label{pb-ex-5.3}
\begin{cases}
\, ( (|x_{1}'|^{2}+|x_{2}'|^{2}) x_{1}' )' = \mathfrak{h}_{1}(x_{1},x_{2}) + \sin\dfrac{\lambda}{\sqrt{1-\lambda^{2}}}\, e_{1}(t),
\\
\, ( (|x_{1}'|^{2}+|x_{2}'|^{2}) x_{2}' )' = \mathfrak{h}_{2}(x_{1},x_{2}) + \sin\dfrac{\lambda}{\sqrt{1-\lambda^{2}}} \, e_{2}(t),
\end{cases}
\end{equation}
where $\mathfrak{h}_{1},\mathfrak{h}_{2}\in\mathcal{C}(\mathbb{R}\times\mathbb{R}, \mathbb{R})$, $e_{1},e_{2}\in\mathcal{C}(\mathopen{[}0,T\mathclose{]},\mathbb{R})$, $\lambda\in(-1,1)$. This problem corresponds to \eqref{eq-example-4} with $\phi(\xi)=\Vert\xi\Vert^{2}\xi$.
It is easy to check that hypothesis \ref{phi=A} holds true. In order to verify hypotheses \ref{hp-A7} and \ref{hp-A8}, for instance we can take
\begin{equation*}
\mathfrak{h}_{1}(x_{1},x_{2}) = (x_{1})^{3} + x_{1}(x_{2})^{2},
\quad
\mathfrak{h}_{2}(x_{1},x_{2}) = (x_{2})^{3} + (x_{1})^{2}x_{2}.
\end{equation*}
Hence,
\begin{equation*}
\langle h_{0}(x_1,x_2),(x_1,x_2) \rangle = (x_{1})^{4} + 2 (x_{1}x_{2})^{2} + (x_{2})^{4}= \|(x_1,x_2)\|^{4} \geq \|(x_1,x_2)\|^{2},
\end{equation*}
for $\|(x_1,x_2)\|\geq 1$. Moreover, $h_{0}$ is odd and thus \ref{hp-A8} holds for every $\delta>0$ and $p=0$.
Theorem~\ref{th-ex-2} ensures that the closure of the continuum emanating from $x=0$ contains either a point $(0,\bar{x})$ with $\bar{x}\neq 0$ or a point $(\pm1,x)$ with $x \in \mathcal{C}^{1}_{T}$.
We can exclude the first situation, analogously to the proof of Theorem~\ref{th-ex-3}, as a consequence of \ref{hp-A4} and \ref{hp-A5} which trivially hold true in this case.

We conclude this example by noting that constructing or identifying a scenario in which the first alternative of Theorem~\ref{th-ex-2} is realised appears to be a challenging task. This indicates that such cases may require highly specific conditions, and that numerical investigations could provide valuable insight into their nature.
\end{example}

\appendix
\section{A topological remark}
\label{appendix-A}

In this appendix, we discuss some complementary questions that arose studying the compactness of the set of solutions of our problems. 

We start by pointing out that:
\begin{quote}
\textit{A continuous map $\mathcal{F}$ between two Banach spaces $X$ and $Y$, which is proper on closed and bounded subsets of the domain (i.e., $\mathcal{F}^{-1}(K)\cap C$ is compact for every $K\subseteq Y$ compact and every $C\subseteq X$ closed and bounded), maintains this property if perturbed by a completely continuous map.}
\end{quote}

We show this fact for a sake of completeness. Let $\Psi \colon X\to Y$ be a completely continuous map, $C$ be a closed and bounded subset of $X$, and $K$ be a compact subset of $Y$. Consider a sequence $(x_{n})_{n}$ in $(\mathcal{F}-\Psi)^{-1}(K)\cap C$. Setting $y_{n} \coloneqq \mathcal{F}(x_{n}) -\Psi(x_{n})$, since $(y_{n})_{n} \subseteq K$, then, up to a subsequence, it converges to $\bar{y}\in K$. Moreover, since the sequence $(x_{n})_{n} \subseteq C$ is bounded and $\Psi$ is completely continuous, thus $(\Psi(x_{n}))_{n}$ converges to $\bar{z} \in Y$, up to a subsequence. The set $\hat{K} \coloneqq \{y_{n} + \Psi(x_{n}) \} \cup \{\bar{y} + \bar{z}\}$ is a compact subset of $Y$ and $\mathcal{F}^{-1}(\hat{K})\cap C$ is compact because $\mathcal{F}$ is proper on $C$. Therefore, $x_{n} = \mathcal{F}^{-1}(y_{n} + \Psi(x_{n}))$ converges up to subsequences, as desired.

Thanks to this fact, the intersection of the set of fixed points of a completely continuous map of $X$  with any bounded and closed subset of $X$ is compact (possibly empty). The reader can observe that this result is used in the proofs of Theorem~\ref{genbifthm} and of Corollary~\ref{corollario1}.

The situation is not the same in the case when a \textit{locally} proper map $\mathcal{F} \colon X\to Y$ is perturbed by a completely continuous or, more generally, by a continuous and locally compact map $\Psi \colon X\to Y$.
In this case, the coincidence set $S=\{x\in X \colon \mathcal{F}(x)=\Psi(x) \}$ is locally compact, but its closed and bounded subsets are not necessarily compact.
An easy counterexample is given by considering a Hilbert basis of a separable Hilbert space, which is locally compact, bounded and complete, but not totally bounded. 

This question has an importance in some nonlinear differential equations in which the dependence on the highest-order derivative of the unknown function is nonlinear, as for example in some neutral differential equations (see \cite{BeFu-06,BFMP-05}).
Let us consider the following family of parameterized nonlinear differential equations 
\begin{equation*}
\begin{cases}
\, x'(t) + f(t,x(t),x'(t)) + g(t,x(t)) = 0,
\\
\, x(0) = x(T),
\end{cases}
\end{equation*}
where $f \colon \mathbb{R} \times \mathbb{R}^{n} \times \mathbb{R}^{n} \to \mathbb{R}^{n}$ and $g \colon \mathbb{R}
\times \mathbb{R}^{n} \to \mathbb{R}^{n}$ are $\mathcal{C}^1$ and continuous, respectively.
The Nemitskii operator $\mathcal{N} \colon \mathcal{C}^1_T \to \mathcal{C}(\mathopen{[}0,T\mathclose{]},\mathbb{R}^n)$, given by 
\begin{equation*}
\mathcal{N}(x) (t) =f(t,x(t),x'(t))+g(t,x(t)), \quad t\in \mathopen{[}0,T\mathclose{]},
\end{equation*}
is not compact due to the presence of $x'$ as a variable, which does not allow the use of any Ascoli--Arzel\`{a} argument. In this case, it is useful to consider
\begin{align*}
&F \colon\mathcal{C}^{1}_T \to \mathcal{C}(\mathopen{[}0,T\mathclose{]},\mathbb{R}^n), \qquad
F(x)(t)= x'(t)+ f(t,x(t),x'(t)),
\\
&G \colon \mathcal{C}^{1}_T\to \mathcal{C}(\mathopen{[}0,T\mathclose{]},\mathbb{R}^n), \qquad
G(x)(t)= g(t,x(t)),
\end{align*}
and the coincidence equation $F(x) = G(x)$.
In this case, suitable conditions on $f$ guarantee that $F$ is a nonlinear Fredholm map of index zero, while $G$ is compact. In this framework concerning nonlinear Fredholm maps a more general degree theory can be applied (cf.~\cite{Va-12} and the references therein). The reader can see \cite{Ca-63, Sm-65} for the local properness of nonlinear Fredholm maps.

Coming back to the topic of this section, the above counterexample of a Hilbert basis of a separable Hilbert space could suggest to add the connectedness as a sufficient condition to obtain the compactness of the investigated set. Unfortunately it is not true. Indeed, consider the classical Hilbert basis $(e_n)_n$ in $\ell^2$, call $\Sigma_n$ the line segment joining $e_n$ and $e_{n+1}$, and define
\begin{equation*}
\Sigma=\bigcup_{n\in\mathbb{N}}\Sigma_n.
\end{equation*}
It is not difficult to observe that $\Sigma$ is locally compact, bounded and closed in $\ell^2$, and, in addition, path connected. Clearly, $\Sigma$ is not compact since it contains a closed and not compact subset (the Hilbert basis).

To obtain the compactness, we can add a \textit{geometric} condition (path connectedness is a topological one). Precisely, we can prove that:
\begin{quote}
\textit{A closed, bounded, locally compact and star-shaped subset $C$ of a Banach space $X$ is compact.}
\end{quote}
To prove this fact, call $y$ a point with respect to which $C$ is star-shaped and consider by contradiction a sequence $(x_n)_n$ in $C$ without convergent subsequences. As a consequence, for a suitable $r>0$, $x_n\notin B[y,r]$ for all $n$ sufficiently large, say $n\geq \hat{n}$, where $B[y,r]$ denotes the closed ball of center $y$ and radius $r$ in $X$.
Without loss of generality, we can take $r$ small enough in such a way that $B[y,r]\cap C$ is compact.
For $n\geq \hat{n}$, let
\begin{equation*}
t_n\coloneqq\Vert x_n-y \Vert, \qquad w_n\coloneqq\frac{x_n-y}{\Vert x_n-y \Vert },
\end{equation*}
and observe that
\begin{equation*}
x_n= y+t_n w_n.
\end{equation*}
The point $v_n\coloneqq y+rw_n$ belongs to $\partial B(y,r)$ and to $C$, since $r < t_{n}$ and $C$ is star-shaped with respect to $y$. Therefore  $(v_n)_n$ admits a convergent subsequence which we still call $(v_n)_n$ to simplify the notation. Consequently, $(w_n)_n$ converges, up to a subsequence. 
Moreover, $(t_n)_n$ is bounded and thus converges, up to a subsequence. It implies that $(x_n)_n$ converges, up to a subsequence, to $\bar{x}\in X$, which belongs to $C$ since this set is closed. This proves the compactness of $C$.

The investigation of more general conditions on $C$ allowing to prove the claim is a challenging problem.

\bigskip

\noindent
\textbf{Acknoledgements.} The arguments in this appendix are the result of fruitful discussions with Prof.~Massimo Furi and Prof.~Fabio Zanolin. We warmly thank them for their contributions.
The authors are also grateful to the anonymous referee for the valuable comments and suggestions that have contributed to enhancing the clarity and quality of the manuscript.

\bibliographystyle{elsart-num-sort}
\bibliography{BeFe-biblio}

\end{document}